\documentclass[12pt]{extarticle}
\usepackage{amsmath, amsthm, amssymb, mathtools, hyperref,color}
\usepackage{amscd}
\usepackage{graphicx}
\usepackage{caption}
\usepackage{subcaption}
\usepackage{tikz}
\usepackage{tikz-cd}
\usetikzlibrary{cd}
\usetikzlibrary{matrix,arrows,decorations.pathmorphing}
\usepackage{enumitem}

\oddsidemargin \evensidemargin
\newtheorem*{theorem*}{Theorem}
\newtheorem{theorem}{Theorem}
\numberwithin{theorem}{section}
\newtheorem{proposition}[theorem]{Proposition}
\newtheorem{lemma}[theorem]{Lemma}
\newtheorem{corollary}[theorem]{Corollary}
\newtheorem{definition}[theorem]{Definition}
\newtheorem{remark}[theorem]{Remark}

\newtheorem{problem}[theorem]{Problem}
\newtheorem{example}[theorem]{Example}

\title{Tropical Graph Curves}
\author{Madhusudan Manjunath\footnote{Part of this work was carried out while the author was affiliated to UC, Berkeley, QMUL London and while visiting the IHES and ICTP. He was supported by the Feoder-Lynen Fellowship of the Humboldt Foundation. He thanks the generous support and the warm hospitality of these institutions.}}
\begin{document}
\maketitle

\begin{abstract}
We study tropical line arrangements associated to a three-regular graph $G$ that we refer to as \emph{tropical graph curves}. Roughly speaking, the tropical graph curve associated to $G$, whose genus is $g$, is  an arrangement of $2g-2$ lines in tropical projective space that contains $G$ (more precisely, the topological space associated to $G$) as a deformation retract. We show the existence of tropical graph curves when the underlying graph is a three-regular, three-vertex-connected planar graph.  
Our method involves explicitly constructing an arrangement of lines in projective space, i.e. a graph curve whose tropicalisation yields the corresponding tropical graph curve and in this case, solves a topological version of the tropical lifting problem associated to canonically embedded graph curves. We also show that the set of tropical graph curves that we construct are connected via certain local operations. These local operations are inspired by Steinitz' theorem in polytope theory.
\end{abstract}
\section{Introduction}
Tropical Geometry provides a framework to translate questions about an algebraic variety to questions about a polyhedral object associated to it called its \emph{tropicalisation}. In its most basic form, the framework is as follows. 

Let $\mathbb{K}$ be a non-archimedean field, i.e. an algebraically closed field with a non-trivial non-archimedean valuation ${\rm val}$ and complete with respect to it.  Let $X$ be a very affine variety over $\mathbb{K}$, i.e. a subvariety of the split torus $(\mathbb{K}^{\star})^n$. The tropicalisation map ${\rm trop}$ takes a point $(p_1,\dots,p_n)$ in $(\mathbb{K}^{\star})^n$ to its coordinatewise valuations $({\rm val}(p_1),\dots,{\rm val}(p_n)) \in \mathbb{R}^n$. The tropicalisation of $X$, denoted by ${\rm trop}(X)$ is then obtained by applying the map ${\rm trop}$ to every point in $X$ and taking the closure with respect to the Euclidean topology on $\mathbb{R}^n$.  

The notion of tropicalisation can be extended in two ways: i. For very affine varieties over arbitrary algebraically ground fields equipped with the trivial valuation, the notion of tropicalisation described above is not satisfactory. In this case, either an alternative description of tropicalisation in terms of initial ideals \cite[Item (2),Theorem 3.2.3]{MacStu15} or a base change to a field with a non-trivial valuation extending this trivial valuation is used \cite[Theorem 3.2.4]{MacStu15} ii.  the notion of tropicalisation has been extended to arbitrary subvarieties of toric varieties (over algebraically closed, valued fields),  referred to as the \emph{Kajiwara-Payne extended tropicalisation} \cite{Pay05}, \cite[Chapter 6]{MacStu15}.

{\bf Tropical Graph Curves:} The protagonists in this paper are \emph{tropical graph curves}. Informally, a tropical graph curve $\mathbb{T}_G$ associated to a three-regular, connected, simple graph $G$ of genus $g$ is an arrangement of $2g-2$ tropical lines in tropical projective space $\mathbb{TP}^{g-1}$ (equipped with the Euclidean topology) that contains $G$ as a deformation retract. Tropical graph curves are  tropical line arrangements in tropical projective space.  Tropical hyperplane arrangements have recently considerable attention in literature, see for example \cite{ArdDev09,JohKam17}.  On the other hand, tropical line arrangements have not received as much attention. We refer to \cite[Theorem C]{BraJonLeeRan18} for a ``universality'' property of tropical line arrangements in the plane. Our main result (see Corollary \ref{tropgra_cor}) can be viewed as a construction of tropical line arrangements that are homeomorphic to a given simple, three-regular, three-connected planar graph. 

Our main motivation for studying tropical graph curves arises from the tropical lifting problem that we introduce in the following.

{\bf Tropical Lifting:} The Bieri-Groves theorem  \cite{BieGro84}, \cite[Chapter 3, Section 3]{MacStu15}, a fundamental theorem in tropical geometry, states that ${\rm trop}(X)$ is a piecewise linear subset of $\mathbb{R}^n$. Hence,  ${\rm trop}(X)$  can be studied via polyhedral geometry and applications of tropical geometry crucially use this polyhedral property. For most of these applications, an understanding of piecewise linear subsets of $\mathbb{R}^n$ that arise as tropicalisations is essential, see \cite{Spe14} for more details. This gives rise to the tropical lifting problem.

\begin{problem}{\rm({\bf Tropical Lifting Problem})} Let $\mathbb{K}$ be an algebraically closed, valued field. Characterise piecewise linear subsets of $\mathbb{R}^n$ (with finitely many pieces) that can be lifted, i.e. obtained as the tropicalisation of a  very affine variety over $\mathbb{K}$ or more generally, as the Kajiwara-Payne extended tropicalisation of a subvariety of a toric variety over $\mathbb{K}$.\end{problem}

%Besides being a problem of intrinsic interest, its importance also stems from applications to enumerative algebraic geometry \cite{Mik05},~\cite{}.  {\color{blue} 10:53 AM, Mentioned before, remove redundancy} 

 The one-dimensional case, i.e. lifting piecewise linear subsets of $\mathbb{R}^n$ of dimension one is already highly non-trivial. Two necessary conditions are that every edge must have \emph{rational slope} and that the set must satisfy the \emph{balancing condition} (also known as the zero-tension condition): there is an assignment of a positive integer called \emph{multiplicity} to each edge such that at every vertex, the sum of the outgoing slopes (where each outgoing slope is represented by a primitive point in $\mathbb{Z}^2$) of the edges incident on it  weighted by the corresponding multiplicity must be zero.  A piecewise linear subset of $\mathbb{R}^n$ satisfying these two necessary conditions is called a \emph{tropical curve} \cite[Section 2]{Mik05} and \cite[Section 1.3]{MacStu15}.   By the \emph{genus} of a tropical curve, we mean its first Betti number when viewed as a metric graph (allowing infinite edge lengths), see \cite[Definition 2.9]{Mik05}. 
 
 The tropical lifting problem for tropical curves is wide open in general. The case of genus zero tropical curves is relatively well understood owing to the work of Mikhalkin \cite[Corollary 3.16]{Mik05} for $n=2$ and to the works of  Nishinou and Siebert \cite[Section 7]{NisSie06} and Speyer \cite[Theorem 3.2]{Spe14} for arbitrary $n$.  Lifting genus one tropical curves was initiated by Speyer \cite[Theorem 3.2]{Spe14}, also see Nishinou \cite[Theorem 2]{Nis09}, Tymokin \cite{Tym12}, Ranganathan \cite[Theorems B and C]{Ran17} for further work.  Katz \cite[Theorem 1.1]{Kat12} introduced necessary conditions for lifting tropical curves of arbitrary genus that generalises Speyer's condition and Nishinou's condition (both for genus one tropical curves and both called ``well-spacedness").

Since the tropical lifting problem is still wide open, studying weaker versions of the problem seems natural. One such weakening is the following: Classify metric graphs $\Gamma$ such that there is a tropical curve $T$ that contains $\Gamma$ as a  deformation retract and can be lifted to a smooth algebraic curve over $\mathbb{K}$. Using the work of Baker, Payne and Rabinoff \cite[Theorem 1.1 and Theorem 5.20]{BakPayRab16} \footnote{These results are stated in terms of ``faithful tropicalisation'', an important notion in the interplay between tropical geometry and non-archimedean geometry.
}, it follows that any metric graph $\Gamma$ whose edge lengths are in the value group of $\mathbb{K}$ satisfies this property with the corresponding tropical curve $T$ being contained in $\mathbb{R}^n$ for a possibly ``high" $n$. Cheung, Fantini, Park and Ulirsch \cite[Theorem 1.2]{CheFanPakUlr15} further refined this result by showing an effective upper bound  on $n$:  the maximum of three and the valence of a vertex $v$ minus one over all vertices $v$ of $\Gamma$. In \cite{CheFanPakUlr15}, the ground field $\mathbb{K}$ is the field of Puiseux series with coefficients in $\mathbb{C}$ and hence, the edge lengths are required to be rational. 
Jell \cite{Jell18} introduced a strengthening of the notion of faithful tropicalisation to so called \emph{fully faithful tropicalisation} and showed that every Mumford curve over  $\mathbb{K}$ admits such a fully faithful tropicalisation.

{\bf Tropical Lifting for Canonical Curves:} From the viewpoint of applications of tropical geometry, lifting to specific classes of algebraic curves is important. One such class is that of \emph{smooth canonical curves}: embeddings of a smooth, proper, non-hyperelliptic algebraic curve into projective space via the global sections of its canonical line bundle. Recall that for any integer $g \geq 3$, a smooth curve in projective space $\mathbb{P}^{g-1}$ is a canonical curve of genus $g$ if and only if it is non-degenerate (not contained in any hyperplane) and has degree $2g-2$ \cite[Theorem 9.3 and Section 9C, Exercise 5]{Eis08}.  We refer to \cite[Chapter 9]{Eis08} for applications of canonical curves. 

The lifting problem of metric graphs to smooth canonical curves takes the following form: 
Classify metric graphs $\Gamma$ such that there is a smooth canonical curve whose tropicalisation  deformation retracts to $\Gamma$.   A classification is wide open, we refer to \cite{BroJosMorStu15} and \cite{HahMarRenTyo18} for progress in the case of genus three metric graphs.   A further weakening leads to a topological version of the problem where only the topological space underlying the metric graph is taken into account. Given an undirected, connected graph $G$ (possibly with multiedges but with no loops), we denote the topological space underlying (any of) its geometric realisations (metric graphs whose underlying graph is $G$) by $G^{\rm top}$.

%tropical curve $T$ that contains $\Gamma$ as a

\begin{problem}{\rm({\bf Topological Tropical Lifting Problem for Smooth Canonical Curves})} Classify graphs $G$ such that there exists a smooth canonical curve whose tropicalisation (in the extended sense) deformation retracts to $G^{\rm top}$. \end{problem}

Even in this topological version, a complete classification is wide open. We refer to \cite[Theorem 3.2]{ChaJir15} for the case when $G$ is the complete graph on four vertices. In the current article, we study the topological tropical lifting problem for certain non-smooth canonical curves, namely canonical embeddings of certain reducible nodal curves called \emph{graph curves}.  We refer to the work of Bayer and Eisenbud \cite{BayEis91} for an introduction to this topic.  

For a simple \footnote{we shall keep this hypothesis throughout the article.}, three-regular, connected graph $G$, the graph curve $X_G$ associated to it is the totally degenerate, nodal curve whose dual graph is $G$.  By totally degenerate, we mean that each irreducible component is isomorphic to the projective line.
The dual graph of a (reducible) curve is the graph whose vertices correspond to its irreducible components and there is an edge between two vertices if their corresponding components intersect. The three-regularity condition on the dual graph $G$ ensures that 
the graph curve $X_G$ is independent of the choice of the nodes (thanks to the three transitivity of the action of the automorphism group of $\mathbb{P}^1$ on $\mathbb{P}^1$).  We address the topological tropical lifting problem for canonical embeddings of the graph curve $X_G$ where $G$ is any three-regular, three-edge-connected planar graph. Before this, note that tropical projective space carries the Euclidean topology (see Subsection \ref{tropproj_sect} for more details) and hence, every extended tropicalisation into it carries the induced topology.  Our main theorem is the following:

\begin{theorem}\label{tropcan_theo} {\rm ({\bf Tropical Lifting for Planar Graph Curves})} Let $\kappa$ be an algebraically closed, valued field.  For every three-regular, three-edge-connected planar graph $G$, there is a canonical embedding of the corresponding graph curve $X_G$ over $\kappa$ whose extended tropicalisation (with respect to the given valuation on $\kappa$) is homeomorphic to $G^{\rm top}$.  \end{theorem} 

To the best of our knowledge, the tropical lifting problem for singular algebraic curves has not been studied before. We say that a canonical embedding of $X_G$ admits a \emph{weakly faithful tropicalisation} or equivalently, that the tropicalisation of this canonical embedding is \emph{weakly faithful} if its tropicalisation (in the extended sense and with respect to the given valuation on $\kappa$) contains $G^{\rm top}$ as a deformation retract. This terminology is justified by the fact that the Berkovich analytification of $(X_G)_{\rm Berk}$ contains $G^{\rm top}$ as a deformation retract.  The space $(X_G)_{\rm Berk}$ is constructed as follows: consider one copy of the Berkovich projective line $(\mathbb{P}^1_{\rm Berk})_u$  for each vertex $u$ of $G$. For each edge $e=(u,v)$ of $G$, note that there is a node $n_e$ of $X_G$, and identify $(\mathbb{P}^1_{\rm Berk})_u$ and  $(\mathbb{P}^1_{\rm Berk})_v$ at the type I points corresponding to $n_e$ \cite[Chapter 4]{Ber90} and \cite{BakPayRab13}.

 The extended tropicalisation of the canonical embedding of $X_G$ promised by Theorem \ref{tropcan_theo}  is an example of a tropical graph curve.  As a corollary,  we obtain the existence of tropical graph curves corresponding to three-regular, three-edge-connected planar graphs.

\begin{corollary} \label{tropgra_cor} Any three-regular, three-edge-connected planar graph has a tropical graph curve associated to it.\end{corollary}

   For a three-regular graph, three-edge-connectivity is equivalent to three-vertex-connectivity \cite[Lemma 2.6]{BayEis91}.  Hence, in this context we will use the term ``three-connected'' for three-edge-connected. In the following, we outline the key steps in the proof of Theorem \ref{tropcan_theo}.

\subsection{Key Ingredients of the Proof of Theorem \ref{tropcan_theo}}

We explicitly construct a canonical embedding of $X_G$ and show that its extended tropicalisation is weakly faithful. We refer to this embedding as the \emph{sch\"on embedding} \footnote{This is not to be confused for sch\"on compactifications \cite[Definition 6.4.19.]{MacStu15}} of $X_G$, denoted by $X_G^{\rm sch}$.  The sch\"on embedding can be described in geometric terms as follows.
Since $G$ is a three-vertex-connected planar graph, by Steinitz' theorem ((\cite[Chapter 4]{Zie07})),  it is the one-skeleton of a three-dimensional polytope $P$. Furthermore, since $G$ is three-regular, the polar $P^{\vee}$ of $P$ is a simplicial polytope.  Consider the Stanley-Reisner surface of the simplicial complex associated to $P^{\vee}$.  The sch\"on embedding is a hyperplane section of this surface.  We refer to Proposition \ref{scho-stanries_theo} for more details. We also refer to Bayer and Eisenbud \cite[Section 6]{BayEis91} where general hyperplane sections of this Stanley-Reisner surface have been studied. 

We study the extended tropicalisation of $X_G^{\rm sch}$ in terms of the primary decomposition of its defining radical ideal. The following explicit description of this primary decomposition plays an important role.  The primary decomposition of the sch\"on embedding is constructed in terms of a planar embedding of $G$.  The Euclidean closure of the unique unbounded component in the complement of $G$ in $\mathbb{R}^2$ is called the \emph{exterior face}. The Euclidean closures of the other components are called \emph{interior faces}. By Euler's formula for planar graphs, there are precisely $g$ interior faces of a planar embedding of $G$, where $g$ is the genus (also known as the first Betti number) of $G$.  We identify the homogenous coordinate ring of $\mathbb{P}_{\kappa}^{g-1}$ with $\kappa[x_F|$~{\rm over all interior faces $F$ of $G$}] (equipped with its standard grading).    Furthermore, any canonical embedding of a graph curve $X_G$ of arithmetic genus $g$ (also, equal to the genus of $G$) is an arrangement of $2g-2$ lines in $\mathbb{P}_{\kappa}^{g-1}$ \cite[Proposition 1.1 (and its proof)]{BayEis91}.  A three-regular graph of genus $g$ has $2g-2$ vertices and $3g-3$ edges.   The $2g-2$ vertices are in bijection with the irreducible components of $X_G$. 

To each vertex $v$ of $G$, we associate a line in  $\mathbb{P}_{\kappa}^{g-1}$ defined by an ideal $L_v$ corresponding to it.  This line is the irreducible component of $X_G$ corresponding to $v$. Note that $L_v$ is given by $g-2$ linearly independent linear forms.  We distinguish between two types of vertices, namely \emph{interior} and the \emph{exterior} vertices.  An interior vertex is a vertex that is not incident on the exterior face. Otherwise, the vertex is called an exterior vertex. Note that  an interior vertex has precisely three interior faces incident on it whereas an exterior vertex has precisely two interior faces incident on it.

For an interior vertex $v$, the line $L_v$ is cut out by the linear form $x_{F_i}+x_{F_j}+x_{F_k}$ where $F_i,F_j,F_k$ are the three interior faces that are incident on it and by the variables $x_F$ over all interior faces $F \notin 
\{ F_i,F_j,F_k\}$. Note that we have specified $g-2$ linearly independent linear forms. For an exterior vertex $v$, the line $L_v$ is generated by the variables $x_F$ over all interior faces $F$ that are not incident on $v$.  Here again, we have specified $g-2$  linearly independent linear forms.

\begin{figure}
  \includegraphics[width=12cm]{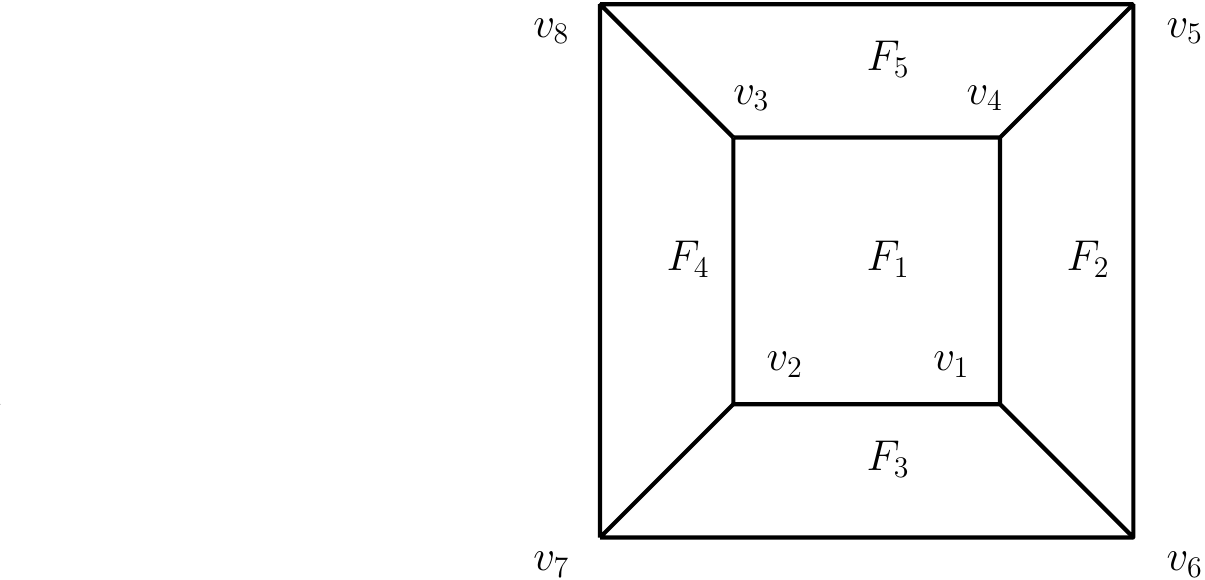}
  \caption{A planar embedding of the cube.}\label{cube}
\end{figure}

\begin{example}  \rm  Consider the one-skeleton of a cube (as shown in Figure \ref{cube}). Here $g=5$ and any canonical embedding of the associated graph curve is an arrangement of eight lines in $\mathbb{P}^4_{\kappa}$.   In this case, $v_1,v_2,v_3$ and $v_4$ are interior vertices and the others are exterior vertices.   The eight lines are the following:

{
 \large
 \begin{center}
$L_{v_1}=\langle x_{F_1}+x_{F_2}+x_{F_3}, x_{F_4}, x_{F_5} \rangle$,\\
$L_{v_2}=\langle x_{F_1}+x_{F_3}+x_{F_4}, x_{F_2}, x_{F_5} \rangle$,\\
$L_{v_3}=\langle x_{F_1}+x_{F_4}+x_{F_5}, x_{F_2}, x_{F_3} \rangle$,\\
$L_{v_4}=\langle x_{F_1}+x_{F_2}+x_{F_5}, x_{F_3}, x_{F_4} \rangle$,\\
$L_{v_5}=\langle x_{F_1}, x_{F_3}, x_{F_4} \rangle$, $L_{v_6}=\langle x_{F_1}, x_{F_4}, x_{F_5} \rangle$,\\
$L_{v_7}=\langle x_{F_1}, x_{F_2}, x_{F_5} \rangle$, $L_{v_8}=\langle x_{F_1}, x_{F_2}, x_{F_3}\rangle$.
\end{center}}
 \qed
\end{example}

We consider the extended tropicalisation ${\rm tropproj}(X_G^{\rm sch})$ of the resulting arrangement of lines. This is an arrangement of $2g-2$ tropical lines $\{{\rm tropproj}(L_v)\}_v$ in $\mathbb{TP}^{g-1}$ that we denote by $\mathcal{T}$.  We refer to Figure \ref{exconn_fig} for examples. We identify $\mathbb{TP}^{g-1}$ with a $(g-1)$-simplex, see Subsection \ref{tropproj_sect} for more details.   We construct a homeomorphism $\phi$ between $G^{\rm top}$  and  $\mathcal{T}$ (Subsection \ref{homeoproof_subsect}). In the following, we identify  key properties of $\mathcal{T}$ that go into the construction of $\phi$. The tropical lines ${\rm tropproj}(L_u)$ and ${\rm tropproj}(L_v)$ intersect if and only if $u$ and $v$ are adjacent in $G$ (Lemma \ref{troplinesinter_prop}).  If $u$ is an interior vertex, then ${\rm tropproj}(L_u)$ contains precisely one branch point and this is of valence three.  If $u$ is an exterior vertex, then ${\rm tropproj}(L_u)$ is an edge of $\mathbb{TP}^{g-1}$ and hence, consists  only of bivalent points. These properties lead to the following classification of the points of $\mathcal{T}$ (Lemma \ref{trivalchar_lem}) that we summarise in the following. The points of $\mathcal{T}$ are either bivalent or trivalent. The trivalent points of $\mathcal{T}$ are exclusively of the following two types: i. Branch point $b_u$ of ${\rm tropproj}(L_u)$ where $u$ is an interior vertex. ii. The intersection point $\zeta_{(u,\iota(u))}$ of ${\rm tropproj}(L_u)$ and ${\rm tropproj}(L_{\iota(u)})$ where $u$ is an exterior vertex and $\iota(u)$ is the unique interior vertex adjacent to it (Item \ref{th:five}, Proposition \ref{grapro_prop}).  With this information at hand, we define $\phi$ on the vertex set of $G^{\rm top}$ as follows: 

\begin{equation}
\phi(u)=
\begin{cases}
b_u,\text{ if $u$  is an interior vertex},\\
\zeta_{(u,\iota(u))}, \text{ if $u$ is an exterior vertex}.
\end{cases}
\end{equation}
With some additional effort, this definition can be extended to $G^{\rm top}$ yielding the homeomorphism $\phi$, we refer to Subsection \ref{homeoproof_subsect} for more details. 

\subsection{Connectivity between Tropicalisations}

Steinitz' theorem (\cite[Chapter 4]{Zie07}, Theorem \ref{Stein_thm}) states that a graph is the one-skeleton of a three-polytope if and only if it is simple, planar and three-vertex-connected. 
A standard proof of this theorem shows a connectivity property of three-vertex-connected planar graphs with respect to an operation called the \emph{$\Delta$$Y$-transformation}.  Motivated by this, we prove a connectivity result between the extended tropicalisations of sch\"on embeddings (Section \ref{conn_sect}).  We define a tropical analogue of the notion of $\Delta$$Y$ (and $Y$$\Delta$) transformations. A \emph{tropical $\Delta$$Y$ transformation} is an operation that transforms a (certain type) of tropical line arrangement to another.  We also define an operation called \emph{contraction-elongation operation} and its tropical analogue, and show the following connectivity property.

\begin{theorem}\label{conn_theo} Let $G_1$ and $G_2$ be three-regular, three-connected planar graphs, and let $\mathcal{T}_{G_1}$ and $\mathcal{T}_{G_2}$ be the extended tropicalisations of $X_{G_1}^{\rm sch}$ and $X_{G_2}^{\rm sch}$, respectively. There exists a finite sequence consisting of tropical $\Delta$$Y$, tropical $Y$$\Delta$ and tropical contraction-elongation transformations that transforms   $\mathcal{T}_{G_1}$ to $\mathcal{T}_{G_2}$. \end{theorem}

One potential application of this connectivity result is in carrying out inductive arguments on the set $\{ \mathcal{T}_G\}_G$.

  % Showing that the sch\"on embedding yields a weakly faithful tropicalisation by tropicalising the minimal generators of $I_{\rm sch}$ seems difficult, the primary decomposition of its defining radical ideal given by the line arrangement $\{ L_v\}_{v}$ seems more tractable from a tropical point of view.  

{\bf A Future Direction:} An approach to tropical lifting for (certain) smooth canonical curves by ``deforming'' the sch\"on embedding of $X_G$. It seems plausible that this deformation can be carried out via deformations of the associated Stanley-Reisner surface.

   {\bf Acknowledgement:}  We thank Bo Lin, Ralph Morrison and Bernd Sturmfels very much, this work has significantly benefitted  from the several discussions we had with them.  We have benefitted very much from our discussions with Lorenzo Fantini, particularly in relation to Berkovich spaces.  We thank  Omid Amini,  Erwan Brugall\'e, Alex Fink, Dhruv Ranganathan and Martin Ulirsch for the helpful discussions.

\section{Preliminaries}
\subsection{A Brief Interlude into Tropical Projective Space}\label{tropproj_sect}

We start by briefly recalling tropical projective space, we refer to \cite[Chapter 6, Section 2]{MacStu15}  for a detailed discussion. Analogous to its classical counterpart, tropical projective space in $n$-dimensions $\mathbb{TP}^n$ can be constructed in different ways,  we describe the one via compactification here.  This mimics the construction of projective space as a torus compactification. 

We consider the hyperplane $(1,\dots,1)^{\perp}$ where $(1,\dots,1) \in \mathbb{R}^{n+1}$ as a model for the tropical torus in $n$-dimensions. Each point in $(1,\dots,1)^{\perp}$ is a representative of an orbit of tropical multiplication of $\mathbb{R}$ on $\mathbb{R}^{n+1}$ via $\lambda \odot (q_0,\dots,q_n)=(\lambda+q_0,\dots,\lambda+q_n)$. We compactify it with the $(n+1)$-coordinate hyperplanes. It is convenient to think of these hyperplanes $\mathcal{H}_0,\dots,\mathcal{H}_{n}$, say as living at ``infinity''. In particular, $\mathcal{H}_i$ is the intersection of the affine copy of the hyperplane $(0,\dots,0,\underbrace{1}_{i},0,\dots,0)^{\perp}\cap (1,\dots,1)^{\perp}$ at ``infinity''. Hence, $\mathbb{TP}^n$ is homeomorphic to the $n$-simplex, also see \cite[Example 6.2.4 and Remark 6.2.5]{MacStu15}.  For each $i$, the $i$-dimensional faces of $\mathbb{TP}^n$ are in bijection with the $i$-dimensional orbits of the standard torus action on $\mathbb{P}^n$.  
 This identification is particularly useful for visualisation purposes.  Note that $\mathbb{TP}^n$ inherits a topology from the Euclidean topology on $\mathbb{R}^n$ that we also refer to as the \emph{Euclidean topology} on  $\mathbb{TP}^n$.

\subsubsection{Tropicalising into $\mathbb{TP}^n$} 

In the following, we briefly discuss tropicalisation of a subvariety of projective space into tropical projective space to fit our needs in the future sections.   We refer to \cite[Subsection 6.2]{MacStu15} for a more thorough treatment of this topic. 

%{\color{blue} 11:57 AM, Revise the definition of extended tropicalisation in terms of the torus orbits.}
Given a graded ideal $I$ of $\kappa[x_1,\dots,x_{n+1}]$ where $\kappa$ is an algebraically closed, valued field.  The \emph{extended (Kajiwara-Payne)  tropicalisation} ${\rm tropproj}(I)$ of $I$ (into $\mathbb{TP}^n$) is defined as the union of the tropicalisations of $I$ when restricted to each torus orbit of $\mathbb{P}^n$. 

Note that each torus orbit of $\mathbb{P}^n$ corresponds to a (possibly) empty subset $\mathcal{V}$ of $[1,\dots,n+1]$. Its coordinate ring is identified with the Laurent polynomial ring $\kappa[x_j^{\pm 1}| j\notin \mathcal{V}]$. The restriction of $I=\langle g_1,\dots,g_r \rangle$ to this torus orbit is the ideal $\tilde{I}=\langle \tilde{g_1},\dots,\tilde{g_r} \rangle$ of  $\kappa[x_j^{\pm 1}| j\notin \mathcal{V}]$ where $\tilde{g_j}$ is obtained from $g_j$ by setting each variable $x_j$ where $j \in \mathcal{V}$ to zero. 
%

%For example, if $f=(x_1^2+x_1x_2+x_1x_3)$ then $m=x_1$ and $g=(x_1+x_2+x_3)$. The tropicalisation of $f$ into $\mathbb{TP}^2$ is the union of the  closure of the tropicalisation of $g$ into $H_0$ and the facet of $\mathbb{TP}^2$ corresponding to $x_1$. In contrast, the tropicalisation of $f$ into $H_0$ is just the tropicalization of $g$ into $H_0$.

%The tropicalisation ${\rm tropproj}(I)$ of a non-zero graded ideal $I$ into $\mathbb{TP}^n$ is defined as:

%\begin{center} 
%${\rm tropproj}(I)=\cap_{f \in I \setminus \{0\}}{\rm tropproj}(f)$.
%\end{center}
Recall that a generating set $B$ of $I$ is called a \emph{tropical basis} for $I$ if $\cap_{f \in B}{\rm trop}(f)={\rm trop}(I)$ \cite[Definition 2.6.3]{MacStu15}. A  homogenous generating set $B$ of $I$ is called an \emph{extended tropical basis} for $I$ if $\cap_{f\in B}{\rm tropproj}(f)={\rm tropproj}(I)$. 

%Every graded ideal $I$ has a finite tropical basis \cite[Theorem 2.6.6]{MacStu15}. If the valuation on the ground field is non-trivial, then the fundamental theorem of tropical geometry holds, i.e. the closure of image of the projective variety cut out by $I$ under the valuation map coincides with ${\rm tropproj}(I)$ \cite[Corollary 6.2.16, Theorem 6.2.15]{MacStu15}. 
%{\color{blue} 11:39 AM, 26 November, 2021: Go over Corollary 6.2.16 more carefully.}

\subsubsection{Linear Subspaces of Tropical Projective Space} 
 
We primarily encounter tropicalisations of linear subspaces of projective space, in particular lines.  

\begin{definition}\cite[Definition 4.2.1]{MacStu15}
A $k$-dimensional tropicalised linear subspace in $\mathbb{TP}^n$ is defined as the extended tropicalisation ${\rm tropproj}(I)$ of the defining ideal $I$ of a $k$-dimensional linear subspace of $\mathbb{P}^{n}$. 
\end{definition}

In order to tropicalise an ideal into $\mathbb{TP}^n$,  knowing an extended tropical basis apriori is particularly useful. The linear subspaces we encounter in this paper are all defined by a set of linear forms with mutually disjoint support. For instance, the ideal $L_{v_2}=\langle x_{F_1}+x_{F_3}+x_{F_4}, x_{F_2}, x_{F_5} \rangle$ from the introduction.  Such a linear subspace has a particularly simple extended tropical basis.

\begin{lemma}\label{tropbas_lem}
Suppose that the ideal $I$ is generated by non-zero linear forms $\ell_1,\dots,\ell_r$ such that their supports are mutually disjoint. The set $\{\ell_1,\dots,\ell_r\}$ is an extended tropical basis for $I$. 
\end{lemma}
\begin{proof}
Recall that a circuit of $I$ is a linear form whose support is inclusion-minimal. By  \cite[Lemma 4.3.16]{MacStu15},  the circuits of $I$ form a tropical basis for it. 
The main idea behind the proof is to apply this lemma to every torus orbit of $\mathbb{P}^n$.  

% (here,  the tropicalisation is into the algebraic torus)

For any torus orbit of  $\mathbb{P}^n$, let $\tilde{\ell}_j$ and $\tilde{I}$ be the restrictions of $\ell_j$ and $I$, respectively to this torus orbit.  Since the supports of $\ell_1,\dots,\ell_r$ are mutually disjoint, they are precisely the set of circuits of $I$.  More generally,  the set of non-zero elements of $\{\tilde{\ell}_1,\dots,\tilde{\ell}_r \}$ are precisely the set of circuits of $\tilde{I}$ and by \cite[Lemma 4.3.16]{MacStu15}, we know that this set is a tropical basis for $\tilde{I}$. Hence, $\{\ell_1,\dots,\ell_r\}$ is an extended tropical basis for $I$.

\end{proof}

\subsection{Remarks on Planar Embeddings}\label{remplaemb_subsect}
In the following, we make precise the sense in which we use the term ``planar embedding'' throughout the article and record facts about them that we use frequently.  Before this,  we note that each edge of $G$ corresponds to an open interval in $G^{\rm top}$.

\begin{definition}\label{planemb_def}
A planar embedding of a simple graph $G$ is a continuous, injective function $\tau:G^{\rm top} \rightarrow \mathbb{R}^2$ such that the following properties are satisfied:
\begin{enumerate}
\item The function $\tau$ takes each edge $e$ of  $G^{\rm top}$ to an open interval. 
\item\label{conv:prop} The set $\mathbb{R}^2 \setminus ({\rm Im}(\tau))$, where ${\rm Im}(\tau)$ is the image of $\tau$, consists of finitely many connected components. All connected components, except precisely one, are bounded.  The (Euclidean) closure of each bounded connected component is a convex polygon. The unbounded component is the complement of a convex polygon. 
\end{enumerate}
\end{definition}

%{\color{blue} \item  Every vertex of $G$ is a vertex of each interior face that is incident on it. }

Recall from the introduction that the Euclidean closures of the bounded components are called the interior faces of the planar embedding and the Euclidean closure of the unbounded component is called the exterior face of the planar embedding. By Steinitz' theorem (\cite[Chapter 4]{Zie07}), every three-vertex-connected planar graph $G$ is a one-skeleton of a three-dimensional polytope and an embedding of $G$ can be obtained via a stereographic projection of this polytope into $\mathbb{R}^2$. We also refer to the related notion of convex embeddings of planar graphs \cite[Chapter 4]{LovVet02}.  A vertex of $G$ that is not incident on the exterior face is called an \emph{interior vertex} and otherwise, the vertex is called an \emph{exterior vertex}.    In the following, we record properties of $\tau$ that will be useful in the forthcoming sections.
\begin{proposition}\label{grapro_prop}
Let $G$ be a planar graph. The following properties hold:
 \begin{enumerate}
% \item\label{th:one}   A planar embedding of $G$ has precisely $g$ interior faces where $g$ is the genus of $G$.
\item\label{th:two} Three distinct vertices that are pairwise adjacent do not share two distinct interior faces.  Three distinct vertices can share at most one face (interior or exterior).  

\item \label{th:six}   If  a pair of distinct vertices are not both exterior vertices and are adjacent, then they share precisely two interior faces. If both are exterior vertices and are adjacent, then they share precisely one interior face of $G$. If a pair of distinct vertices  are not adjacent, then they share at most one interior face of $G$. Furthermore, if both these vertices are exterior vertices, then they do not share an interior face. Every edge is shared by precisely two faces (one of which might be the exterior face).

%\item \label{th:three}

\item \label{th:four} If $G$ is a three-regular, three-connected graph, then every interior face can share at most one edge with the exterior face. 

\item \label{th:five} If $G$ is three-regular, then every exterior vertex has a unique interior vertex adjacent to it.

%\item \label{th:sev} 
%\item \label{th:eig}  
\end{enumerate}
\end{proposition}

The proof of Proposition \ref{grapro_prop} mainly uses the convexity property of the faces (Item \ref{conv:prop}, Definition \ref{planemb_def}). 

%{\color{blue} 11:35 AM, 26 March, 2022: Ensure that the closure condition in the definition of the faces (interior and exterior) is incorporated everywhere.}

%\section{Existence and Realisability of Tropical Graph Curves}\label{extrea_sect}

\section{The Sch\"on Embedding of $X_G$}

We begin by recalling the sch\"on embedding of a graph curve $X_G$ where $G$ is a  three-regular, three-connected planar graph.   Let $G$ be a three-regular, three-connected planar graph.   %Given a planar embedding of $G$, the unique unbounded component in the complement of $G$ in $\mathbb{R}^2$ is called the exterior face and the other components are called interior faces.
 We label the interior faces of the planar embedding by variables: let $x_F$ be the variable corresponding to the face $F$. Let $R$ be the graded polynomial ring with coefficients in $\kappa$ (the ground field) and variables $x_F$ where $F$ ranges over all the interior faces of the planar embedding of $G$. We identify $\mathbb{P}^{g-1}$ with ${\rm Proj}(R)$.  
 
 An interior vertex is a vertex that is not incident on the exterior face and otherwise, the vertex is called an exterior vertex.   To each vertex $v$ of $G$, we associate an ideal $L_v$ defined by a collection of linear forms as follows: 

\begin{enumerate}

\item If $v$ is an interior vertex, then $L_v$ is the ideal generated by $x_F$ over all interior faces $F$ not incident on $v$ and the sum of the variables corresponding to the three interior faces incident on $v$.

\item If $v$ is an exterior vertex, then $L_v$ is the ideal generated by $x_F$ over all interior faces $F$ not incident on $v$. %Note that there are precisely $g-2$ such faces.

\end{enumerate}

In both the cases above, $L_v$ defines a line in $\mathbb{P}^{g-1}$.  We refer to the algebraic curve in $\mathbb{P}^{g-1}$ corresponding to the line arrangement defined by $L_v$ as $v$ varies over all the vertices of $G$ as the \emph{sch\"on embedding} of the graph curve $X_G$.  As we shall see in Proposition \ref{schocan_prop}, this is a canonical embedding of $X_G$.

\begin{proposition}\label{dualgranondeg_prop} The dual graph of the sch\"on embedding of $X_G$ is $G$. Furthermore, the sch\"on embedding is non-degenerate, i.e. it is not contained in any hyperplane of $\mathbb{P}^{g-1}$.  \end{proposition}

\begin{proof}
The first part of the proposition follows by noting that if vertices $u \neq v$ are adjacent, then, by Item \ref{th:six}, Proposition \ref{grapro_prop}, they share precisely two distinct faces $F_i$ and $F_j$ (one of which is the exterior face precisely when both $u$ and $v$ are exterior vertices). Furthermore, if at least one of $u$ or $v$ is an interior vertex, then $L_u+L_v=\langle x_{F_i}+x_{F_j},~x_F|$~{\rm $F$ is not incident on either $u$ or $v$} $\rangle$  and otherwise, $L_u+L_v=\langle x_F|$~{\rm $F$ is not incident on either $u$ or $v$}$\rangle$. In both cases, $L_u+L_v$ defines a point in $\mathbb{P}^{g-1}$. Conversely, if vertices $u \neq v$ are not adjacent, then, by Item \ref{th:six}, Proposition \ref{grapro_prop},  $L_u+L_v$ is the irrelevant ideal. Hence, the dual graph of the sch\"on embedding is $G$.

For the second part, suppose for the sake of contradiction that a hyperplane defined by the linear form $\sum_{F}a_Fx_F$ contains the sch\"on embedding. Hence, $\sum_{F}a_Fx_F \in L_v$ for all vertices $v$.  For each interior vertex $v$, the condition that $\sum_{F}a_Fx_F \in L_v$ implies that the coefficients $a_{F_i}=a_{F_j}=a_{F_k}$ for the three interior faces $F_i,F_j$ and $F_k$ that are incident on $v$.    By Steinitz' theorem, $G$ is the one-skeleton of a three-dimensional polytope $P$. Note that the polar polytope $P^{\vee}$ of $P$ is also a three-dimensional polytope. Its vertices are in bijection with the faces of $G$ (including the exterior one).  By Steinitz' theorem, the one-skeleton of $P^{\vee}$ is three-vertex-connected. Hence, the graph obtained from the one-skeleton of $P^{\vee}$ by deleting the vertex corresponding to the exterior face of $G$ is connected and this implies that all the coefficients  $a_F$ are equal. Hence, the hyperplane must be defined by $\sum_{F}x_F$. However, this hyperplane does not contain the line corresponding to $L_v$ for any exterior vertex $v$.

\end{proof}

 In the following, we describe the sch\"on embedding in terms of certain Stanley-Reisner ideals. Recall that the Stanley-Reisner ideal of a simplicial complex $\Delta$ with vertex set $\{1,\dots,n\}$ is a monomial ideal $I_{\Delta}$ in $\kappa[x_1,\dots,x_n]$, \cite{FraMerSch14}.  It is generated by products of variables $\prod_{i \in \bar{F}} x_i$ where $\bar{F} \subseteq \{1,\dots,n\}$ is a non-face of $\Delta$.  We refer to its associated projective variety as the Stanley-Reisner variety of $\Delta$.

Recall that for a polytope $Q$, its dual simplicial complex is the simplicial complex whose vertex set is the set of facets of $Q$ and the simplices are the subsets consisting of facets of $Q$ whose intersection is non-empty.  Let $M$ be the dual simplicial complex of the three-dimensional polytope $P$ associated to $G$. Let $\mathcal{F}$ be the set of faces (both interior and exterior) of the planar embedding of $G$.  Note that the vertices of $M$ are in bijection with the facets of $P$ and the facets of $P$ are in turn in bijection with the elements of $\mathcal{F}$. Hence, we consider the graded polynomial ring $\tilde{R}$ with coefficients in $\kappa$,  with variables $x_F$ where $F$ varies over $\mathcal{F}$.  We identify $\mathbb{P}^{g}$ with ${\rm Proj}(\tilde{R})$.  Since $M$ is a two-dimensional simplicial complex,  the quotient ring $\tilde{R}/I_{M}$ of its Stanley-Reisner ideal has Krull dimension three \cite[Corollary 1.15]{MilStu05}, \cite[Theorem 6.15]{FraMerSch14}. Hence, it defines a surface $S$ in $\mathbb{P}^g$ referred to as the \emph{Stanley-Reisner surface of $M$}. Furthermore, we identify  ${\rm Proj}(R)$ with the hyperplane $\sum_{F \in \mathcal{F}} x_F$ of $\mathbb{P}^g$.

     %We regard the Stanley-Reisner variety of $M$ as a subvariety of $\mathbb{P}^{g}$ and identify  ${\rm Proj}(R)$ with the hyperplane $\sum_{f \in \mathcal{F}} x_f=0$.

 In the following proposition, we show that the sch\"on embedding is the intersection of the Stanley-Reisner surface of $M$ with the hyperplane $\sum_{F \in \mathcal{F}} x_F$. 
 For a homogenous ideal $I$, let $V(I)$ be the projective variety defined by it. 

\begin{proposition}\rm{(\bf Sch\"on Embedding in terms of the Stanley-Reisner Surface)}\label{scho-stanries_theo}
 Let $I_{\rm sch}$ be the ideal  of $R$ generated by polynomials obtained by replacing the variable $x_E$, corresponding to the exterior face, in each monomial minimal generator of $I_M$ by $-\sum_{F \neq E} x_F$.  The projective variety $V(I_{\rm sch})$  is the sch\"on embedding of $X_G$. 
\end{proposition}

\begin{proof} %The sch\"on embedding of $X_G$ is, by construction, an arrangement of $2g-2$ lines in $\mathbb{P}^{g-1}$.  We verify that the ideal  $J_M$ also cuts out an arrangement of $2g-2$ lines in $\mathbb{P}^{g-1}$ as follows. 
The Stanley-Reisner ideal $I_M$ is a radical ideal and hence, its primary decomposition is the intersection of its associated primes.  By \cite[Proposition 4.11]{FraMerSch14},  it has the primary decomposition $\cap_{v \in V(G)}\mathcal{P}_v$ where $V(G)$ is the set of vertices of $G$ and $\mathcal{P}_v$ is the ideal generated by the variables $x_F$ where $F\in \mathcal{F}$ is not incident on $v$.  Hence, the associated Stanley-Reisner surface of $M$ is an arrangement of $2g-2$ two-dimensional planes in $\mathbb{P}^g$. Note that, for each $v$, the intersection of $V(\mathcal{P}_v)$ with  the hyperplane $\sum_{F \in \mathcal{F}} x_F$ is precisely $L_v$, via the identification of ${\rm Proj}(R)$ with  the hyperplane $\sum_{F \in \mathcal{F}} x_F$. Hence, $V(I_M+\langle \sum_{F \in \mathcal{F}} x_F \rangle)=V(I_{\rm sch})$ is the sch\"on embedding of $X_G$. 
\end{proof}

\begin{remark}
\rm To the best of our knowledge, it is not known whether $I_{\rm sch}$ is a radical ideal, i.e. if $I_{\rm sch}=\cap_{v \in V(G)}L_v$ or not. \qed
\end{remark}
Note that by \cite[Corollary 2.2]{BayEis91}, the canonical bundle of $X_G$ is very ample.  Next, we deduce using Proposition \ref{scho-stanries_theo} that the sch\"on embedding of $X_G$ is a canonical embedding.

\begin{proposition}\label{schocan_prop} The sch\"on embedding is a canonical embedding of $X_G$, i.e. an embedding by the complete linear series associated to the canonical bundle of $X_G$.   \end{proposition}

\begin{proof}
By Proposition \ref{dualgranondeg_prop}, the dual graph of the sch\"on embedding is isomorphic to $G$. The rest of the proof is based on the proof of \cite[Corollary 6.2]{BayEis91} where an analogous statement for a general hyperplane section of $S$ is shown. Since the boundary of the polar polytope of $P$ is a geometric realisation of $M$, the simplicial complex $M$ is homeomorphic to a $2$-sphere.  Hence, by \cite[Theorem 6.1]{BayEis91} the Stanley-Reisner surface of $M$ has a trivial canonical bundle. By the adjunction formula \cite[Proposition 30.4.8]{Vak}, we know that $\omega_{G} \cong (\omega_S \otimes_{\mathcal{O}_S} \mathcal{O}_S(X_G))|_{X_G}= (\omega_S \otimes_{\mathcal{O}_S} \mathcal{O}_S(1))|_{X_G}$ where $\omega_G$ and $\omega_S$ are the canonical bundles of $X_G$ and $S$ respectively.  Hence, we conclude that $\omega_G \cong \mathcal{O}_S(1)|_{X_G}$ and hence, the sch\"on embedding is an embedding by a linear series of $\omega_G$. Finally, we note that by Proposition \ref{dualgranondeg_prop}, the sch\"on embedding is non-degenerate and that $h^0(X_G,\omega_G)=g$ \cite[Proposition 1.1]{BayEis91} to conclude that the sch\"on embedding is an embedding by the complete linear series of $\omega_G$. 
\end{proof}
As corollary, we obtain the following. 
\begin{corollary} The sch\"on embedding of $X_G$ is independent of the choice of planar embedding of $G$. \end{corollary}

The following proposition determines an extended tropical basis for the sch\"on embedding and will not be used subsequently.  We include it for possible future applications.

\begin{proposition}\rm{({\bf Tropical Basis of the Sch\"on Embedding})}\label{tropbas_theo}
 The minimal generating set $\mathcal{G}$ of $I_{\rm sch}$ that is presented in Proposition \ref{scho-stanries_theo} is an extended tropical basis.
\end{proposition}
\begin{proof}
Let $\mathcal{G}=\{g_1,\dots,g_r\}$.  Suppose for the sake of contradiction  that $\mathcal{G}$ is not an extended tropical basis, then there is a point $p \in \cap_{j=1}^{r} {\rm tropproj}(g_j)$ that is not contained in the extended tropicalisation of $I_{\rm sch}$. Since the elements in $\mathcal{G}$ are all products of linear forms, this implies that there is a choice of linear forms $\ell_1,\dots,\ell_r$ such that $\ell_j | g_j$ for each $j$ from one to $r$ and such that  $p \in \cap_{j=1}^{r} {\rm tropproj}(\ell_j)$. 

Consider the ideal generated by the linear forms $\ell_1,\dots,\ell_r$. This ideal contains $I_{\rm sch}$.  Since it is generated by  linear forms, its zero locus, being non-empty, is either a point or an irreducible component of the sch\"on embedding.  Furthermore, any $\ell_j$ is either a variable or of the form $\sum_{F \in \mathcal{F}, F \neq E}x_F$ where $E$ is the exterior face.  We claim that $\ell_1,\dots,\ell_r$ is an extended tropical basis.  If $\ell_1,\dots,\ell_r$ are all variables, then this is immediate (also, see Lemma \ref{tropbas_lem}).  Otherwise, we may assume that $\ell_r=\sum_{F \in \mathcal{F}, F \neq E}x_F$ and $\ell_1,\dots,\ell_{r-1}$ are all variables. Let $\ell'_r=\sum_{F \in \mathcal{F}, F \neq E,x_F \notin \{\ell_1,\dots,\ell_{r-1}\}}x_F$. By the definition of extended tropicalisation, we have $\cap_{j=1}^{r} {\rm tropproj}(\ell_j)=  \cap_{j=1}^{r-1} {\rm tropproj}(\ell_j) \cap {\rm tropproj}(\ell'_r)$.  By Lemma \ref{tropbas_lem}, the set $\{\ell_1,\dots,\ell_{r-1},\ell'_r\}$ is an extended tropical basis for $\langle \ell_1,\dots,\ell_r \rangle$ and hence, so is $\{\ell_1,\dots,\ell_r\}$. This implies that $p$ is contained in the extended tropicalisation of $\langle \ell_1,\dots,\ell_r \rangle$ and hence, in the extended tropicalisation of the sch\"on embedding.  This is a contradiction.

 %By eliminating all the variables {\color{blue} 12:33 PM, 28 December, 2021: Check ``the'' once more} that have appeared in this set from  $-\sum_{F \in \mathcal{F}, F \neq E}x_F$, we may assume that $\{\ell_1,\dots,\ell_r\}$ has mutually disjoint support.  
 
\end{proof}

\begin{example}\rm
Let $\mathcal{C}$ be the one-skeleton of the three-dimensional cube, as shown in Figure \ref{cube}. The minimal generating set of the sch\"on embedding of $X_{\mathcal{C}}$ described in Proposition \ref{scho-stanries_theo} is the following. 

{\large
\begin{center}
$\langle x_{F_2}x_{F_4}, x_{F_3}x_{F_5}, x_{F_1}^2+x_{F_1}x_{F_2}+x_{F_1}x_{F_3}+x_{F_1}x_{F_4} +x_{F_1}x_{F_5}\rangle.$
\end{center}}

Hence, the sch\"on embedding is a complete intersection cut-out by three (degenerate) quadrics in $\mathbb{P}^4$  and according to Proposition \ref{tropbas_theo}, these quadrics also form an extended tropical basis. However, canonical embeddings of graph curves are not, in general, complete intersections.

\qed
\end{example}

\section{Tropicalisation of the Sch\"on Embedding}\label{tropschoen_sect}

In the following, we study the extended tropicalisation of the sch\"on embedding of $X_G$ when $G$ is a three-regular, three-connected planar graph. Recall that this extended tropicalisation is contained in $\mathbb{TP}^{g-1}$ and that $\mathbb{TP}^{g-1}$ is homeomorphic to the $(g-1)$-simplex. As in the previous section, we identify $\mathbb{P}^{g-1}$ with ${\rm Proj}(R)$ where $R=\kappa[x_F|~$ $F$ {\rm ranges over the interior faces of the planar embedding of }$G]$ equipped with the standard grading.  Note that each facet of $\mathbb{TP}^{g-1}$  corresponds to a coordinate hyperplane $x_F$ where $F$ is an interior face of the planar embedding of $G$.  We label each facet of $\mathbb{TP}^{g-1}$ with the corresponding interior face $F$. More generally, we label each $i$-dimensional face of $\mathbb{TP}^{g-1}$ for $i \in [0,\dots,g-2]$ with the union of the labels of each facet containing it. Note that for faces $\mathfrak{f}$ and $\mathfrak{g}$ of $\mathbb{TP}^{g-1}$,  we have $\mathfrak{f} \subseteq \mathfrak{g}$  if and only if the label of $\mathfrak{f}$ is contained in the label of $\mathfrak{g}$.  For an algebraically closed field $\mathbb{K}$ with a non-trivial valuation, the homogenous coordinates on $\mathbb{P}_{\mathbb{K}}^{g-1}$, via the extended tropicalisation map, induce coordinates on $\mathbb{TP}^{g-1}$.  We refer to \cite[Section 6.2]{MacStu15} for more details.  In the following, we specify points in $\mathbb{TP}^{g-1}$ by corresponding points in $\mathbb{P}_{\mathbb{K}}^{g-1}$. 

%{\color{blue} Mention that an $i$-dimensional face of $\mathbb{TP}^{g-1}$ is labelled by a set of cardinality $g-i-1$} 

%We also fix coordinates on $\mathbb{TP}^{g-1}$ as follows. We identify the relative interior of each $i$-dimensional face for $i \in [0,\dots,g-1]$ with $i$-dimensional Euclidean space via a homeomorphism and transfer coordinates from the Euclidean space via this homemorphism.  {\color{blue} UNDER CONSTRUCTION: 12:07 PM, 13 January, 2022: Skip this or adapt from a previous draft?} 

{\bf Tropicalisation of the Irreducible Components:} Recall that each irreducible component of the sch\"on embedding is a line defined by the ideal $L_v$.  If $v$ is an interior vertex incident on faces $F_i,~F_j$ and $F_k$  then $L_v=\langle x_{F_i}+x_{F_j}+x_{F_k},x_F|~F$ is an interior face of $G$ not incident on $v \rangle$. The extended tropicalisation ${\rm tropproj}(L_v)$ of $L_v$ is contained in the two-dimensional face $D_v$ of $\mathbb{TP}^{g-1}$ labelled by $\{F|~F$ is an interior face of $G$ not incident on $v\}$. Furthermore,  ${\rm tropproj}(L_v)$ contains precisely one branch point (of valence three). The three branches are labelled by $\{F_i,F_j\},~\{F_i,F_k\}$ and $\{F_j,F_k\}$ according to the sum of the pairs of the corresponding variables that is the initial form along that branch \cite[Definition 3.1.1, Theorem 3.1.3]{MacStu15}.  Each such pair of faces shares a unique edge and hence, a branch is also labelled by that edge $e$, say. The  branch with the label $\{F_s,F_t\}$ intersects the edge $\xi$ of $\mathbb{TP}^{g-1}$ with the label $\{F|~F$ is an interior face that is neither $F_s$ nor $F_t$$\}$ at precisely one point. This point of intersection $\zeta_e$, also denoted by $m_{\xi}$, is the tropicalisation of the point in $\mathbb{P}_{\mathbb{K}}^{g-1}$ with coordinates $(p_F)_F$ such that  $p_{F_s},~p_{F_t}$ are both not zero, ${\rm val}(p_{F_s})={\rm val}(p_{F_t})$ and $p_F=0$ for all interior faces $F$ that is neither $F_s$ nor $F_t$. Hence, the point of intersection lies in the relative interior of this edge $\xi$.  We refer to Figure \ref{troplinesnew} for an illustration of ${\rm tropproj}(L_v)$.

\begin{figure}
  \centering
  \includegraphics[width=8cm]{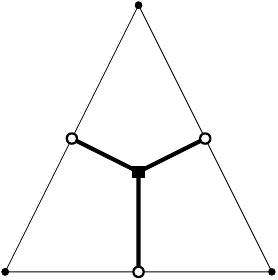}
  \caption{The extended tropicalisation of $L_v$ where $v$ is an interior vertex is shown in thick lines and the two-dimensional $D_v$ is shown in thin lines. The branch point is depicted by the square dot and the intersection points of ${\rm tropproj}(L_v)$ with the boundary of $D_v$ are depicted by the hollow circular dots.}\label{troplinesnew}
\end{figure}

Next, we turn to ${\rm tropproj}(L_v)$ where $v$ is an exterior vertex. In this case, $L_v=\langle x_F|~F$ is an interior face of $G$ not incident on $v\rangle$ and hence, ${\rm tropproj}(L_v)$ is equal to the edge $\xi_v$ of $\mathbb{TP}^{g-1}$ with the label $\{F|~F$ is an interior face of $G$ not incident on $v\}$.

In the following lemma, we determine the points of intersection between the tropical lines ${\rm tropproj}(L_u)$ and ${\rm tropproj}(L_v)$ where $u,~v$ are distinct vertices of $G$. 

\begin{lemma}\label{troplinesinter_prop}
Let $u$ and $v$ be distinct vertices of $G$. The intersection ${\rm tropproj}(L_u) \cap {\rm tropproj}(L_v) \neq \emptyset$ if and only if $u$ and $v$ are adjacent in $G$. If ${\rm tropproj}(L_u) \cap {\rm tropproj}(L_v) \neq \emptyset$, then it is a point and 
this intersection point is as follows: 
\begin{itemize}
\item If $u$ and $v$ are not both exterior vertices, then ${\rm tropproj}(L_u) \cap {\rm tropproj}(L_v)$  is contained in the relative interior of the edge of $\mathbb{TP}^{g-1}$ that is labelled by $\{F|~F$ is an interior face that does not contain both $u$ and $v\}$.
\item If both $u$ and $v$ are exterior vertices, then  ${\rm tropproj}(L_u) \cap {\rm tropproj}(L_v)$ is the vertex of $\mathbb{TP}^{g-1}$  that is labelled by $\{F|~F$ is an interior face that does not contain both $u$ and $v\}$. 
\end{itemize}
\end{lemma}

\begin{proof}

In the following, we invoke Item \ref{th:six}, Proposition \ref{grapro_prop}. 
 Suppose that $u$ and $v$ are interior vertices and are adjacent. Let $e=(u,v)$. In this case, the  two-dimensional faces $D_u$ and $D_v$ of $\mathbb{TP}^{g-1}$, that contain ${\rm tropproj}(L_u)$ and ${\rm tropproj}(L_v)$ respectively,  share an edge $\mu_{u,v}$ that is labelled by the set of interior faces that do not contain both $u$ and $v$. Both ${\rm tropproj}(L_u)$ and ${\rm tropproj}(L_v)$ intersect  $\mu_{u,v}$ at its relative interior, more precisely at the point $\zeta_e$ (recall from the text preceding this lemma).  This is their only point of intersection. We refer to A, Figure \ref{troplinesinter} for a depiction. If one of the two vertices, $u$ say is an interior vertex and the vertex $v$ is an exterior vertex, then recall that ${\rm tropproj}(L_u)$ is contained in $D_u$ and ${\rm tropproj}(L_v)$ is equal to the edge $\xi_v$ of $\mathbb{TP}^{g-1}$.  If $u$ and $v$ are adjacent and share an edge $e$, then the label of $D_u$ is contained in the label of $\xi_v$. Hence, $\xi_v$ is an edge of $D_u$ and ${\rm tropproj}(L_u)$ intersects ${\rm tropproj}(L_v)$ precisely at the point $\zeta_e$.  If both $u$ and $v$ are exterior vertices, then ${\rm tropproj}(L_u)$ and ${\rm tropproj}(L_v)$ are the edges $\xi_u$ and $\xi_v$, respectively. We refer to B, Figure \ref{troplinesinter} for an illustration.  If $u$ and $v$ are adjacent, then they share precisely one interior face $F_{u,v}$. The edges $\xi_u$ and $\xi_v$ share a vertex whose label is the set of all interior faces apart from $F_{u,v}$. This is the only intersection point of ${\rm tropproj}(L_u)$ and ${\rm tropproj}(L_v)$. We refer to C, Figure \ref{troplinesinter} for an illustration.

If $u$ and $v$ are not adjacent, then, by Item \ref{th:six}, Proposition \ref{grapro_prop}, they share at most one interior face. If $u$ and $v$ are both interior vertices, then, based on whether $u$ and $v$ share an interior face or not, the two-dimensional faces $D_u$ and $D_v$ either share precisely one vertex  or are disjoint. Since ${\rm tropproj}(L_u)$ intersects each edge of $D_u$ in the relative interior of that edge, we conclude that ${\rm tropproj}(L_u)$ and ${\rm tropproj}(L_v)$  are disjoint. If $u$ is an interior vertex and $v$ is an exterior vertex, then $\xi_v$ is not an edge of $D_u$, more precisely based on whether $u$ and $v$ share an interior face or not, $D_u$ and $\xi_v$ intersect at a vertex or are disjoint.   As in this previous case, we conclude that ${\rm tropproj}(L_u)$ and ${\rm tropproj}(L_v)$  are disjoint. If $u$ and $v$ are both exterior vertices, then they do not share an interior face of $G$. Hence, $\xi_u={\rm tropproj}(L_u)$ and $\xi_v={\rm tropproj}(L_v)$ are disjoint.
\end{proof}

\begin{figure}
  \centering
  \includegraphics[width=14cm]{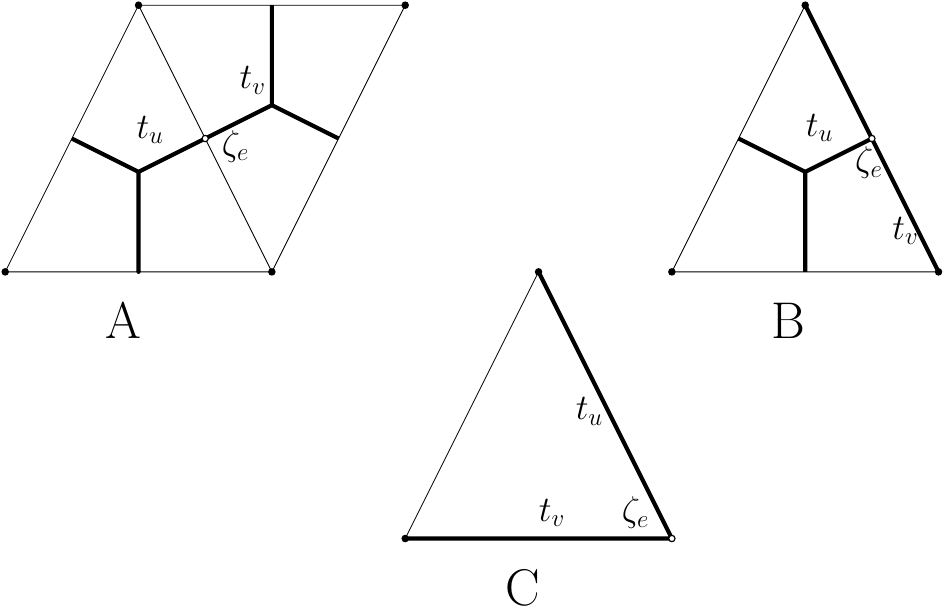}
  \caption{{\bf Intersection of $t_u={\rm tropproj}(L_u)$ and $t_v={\rm tropproj}(L_v)$:} Figures A,~B and C illustrate the cases where (A) $u$ and $v$ are both interior vertices, (B) $u$ is an interior vertex and $v$ is an exterior vertex and  (C) $u$ and $v$ are both exterior vertices, respectively.}\label{troplinesinter}
\end{figure}

Next, we use Lemma \ref{troplinesinter_prop} to classify the trivalent points  of the extended tropicalisation $\mathcal{T}$ of the sch\"on embedding of $X_G$. Before this, we note that in any planar embedding of a three-regular graph, every exterior vertex has precisely one interior vertex adjacent to it.   The following lemma will play a key role in constructing a homeomorphism between $G^{\rm top}$ and the extended tropicalisation of the sch\"on embedding of $X_G$.

\begin{lemma}\label{trivalchar_lem}
The points of $\mathcal{T}$ are either bivalent or trivalent. The trivalent points of $\mathcal{T}$ are of the following two distinct types:
\begin{itemize}
\item A branch point of ${\rm tropproj}(L_u)$ where $u$ is an interior vertex. 

\item An intersection point of ${\rm tropproj}(L_u)$ and ${\rm tropproj}(L_v)$ where $u$ is an exterior vertex and $v$ is the unique interior vertex adjacent to $u$. 
\end{itemize}
\end{lemma}

\begin{proof}
We start by noting that points in each tropical line are either bivalent or trivalent as points in that line. If $u$ is an interior vertex, then the tropical line ${\rm tropproj}(L_u)$ contains a branch point and this is its only trivalent point. Otherwise, ${\rm tropproj}(L_u)$ does not contain a trivalent point.  Furthermore, by Lemma \ref{troplinesinter_prop}, ${\rm tropproj}(L_u)$ (when $u$ is an interior vertex) does not intersect any other tropical line at its branch point. Hence, each such branch point remains a trivalent point as a point in $\mathcal{T}=\cup_u {\rm tropproj}(L_u)$. Any other trivalent point of $\mathcal{T}$ must be an intersection point of two distinct tropical lines. By Lemma \ref{troplinesinter},  the intersection point $\zeta_e$, where $e=(u,v)$, of ${\rm tropproj}(L_u)$ and ${\rm tropproj}(L_v)$ where $u$ is an interior vertex and $v$ is an exterior vertex is a trivalent point of ${\rm tropproj}(L_u) \cup {\rm tropproj}(L_v)$, see B, Figure \ref{troplinesinter}. In the following, we show that $\zeta_e$ is not contained in ${\rm tropproj}(L_w)$ for $w \notin \{u,v\}$. 
Suppose the contrary,  by Lemma \ref{troplinesinter_prop}, we deduce that $w$ is adjacent to both $u$ and $v$, and $w$ is contained in the two interior faces that are shared by $u$ and $v$.  By Item \ref{th:two}, Proposition \ref{grapro_prop},  this is a contradiction.  Hence, $\zeta_e$ is a trivalent point of $\mathcal{T}$. 

%Reasoning for obtaining the previous contradiction: Suppose that these two interior faces that contain $u,~v$ and $w$ are $F_i$ and $F_j$. Since the faces $F_i$ and $F_j$ are both convex polygons (Definition \ref{planemb_def}), both contain the convex hull of $\{u,v,w\}$. Since $u,v$ and $w$ are pairwise adjacent, they are in general position and hence, their convex hull is a triangle. 

Next, we show that any other point in $\mathcal{T}$ that is an intersection point of tropical lines is bivalent. Consider the intersection point $\zeta_e$ of ${\rm tropproj}(L_u)$ and ${\rm tropproj}(L_v)$ where $u$ and $v$ are both interior vertices, see A, Figure \ref{troplinesinter}. The point $\zeta_e$ is a bivalent point of ${\rm tropproj}(L_u) \cup {\rm tropproj}(L_v)$. Suppose that, for the sake of contradiction, $\zeta_e$ is a point of higher valence in $\mathcal{T}$. This implies that there is a vertex $w$ apart from $u$ and $v$ such that ${\rm tropproj}(L_w)$ contains $\zeta_e$. By Lemma \ref{troplinesinter_prop},  we deduce that the vertices $u,~v$ and $w$ share two distinct interior faces and this is a contradiction, Item \ref{th:two}, Proposition \ref{grapro_prop}. Consider the case where $u$ and $v$ are both exterior vertices that are adjacent. We note that the intersection point $\zeta_e$ of ${\rm tropproj}(L_u)$ and ${\rm tropproj}(L_v)$, as shown in C, Figure   \ref{troplinesinter}, is a bivalent point of ${\rm tropproj}(L_u) \cup {\rm tropproj}(L_v)$. We show that it cannot be contained in any other tropical line ${\rm tropproj}(L_w)$.  Suppose the contrary, by Lemma \ref{troplinesinter_prop}, $w$ is adjacent to both $u$ and $v$, and must be an exterior vertex (since $\zeta_e$ is a vertex of $\mathbb{TP}^{g-1}$ and ${\rm tropproj}(L_q)$ contains a vertex  of $\mathbb{TP}^{g-1}$ precisely when $q$ is an exterior vertex). Furthermore, Lemma \ref{troplinesinter_prop} also implies that $u,~v$ and $w$ share an interior face. Since the vertices $u,~v$ and $w$ must be in general position, their convex hull is a triangle $\delta$  and this triangle must be their common interior face.  The graph $G$ is three-regular. Hence, $u$ has another vertex $u_n$, say apart from $v$ and $w$ that is adjacent to it. This vertex $u_n$ must also be incident on the exterior face (since $u_n$ shares two faces with $u$ and is not incident on the triangle $\delta$)  implying that $u$ has three distinct exterior vertices adjacent to it. Since $u$ is an exterior vertex, this is a contradiction. 
\end{proof}

% The tropicalisation of lines corresponding to exterior vertices is not branched (consists only of points of valence two) whereas the tropicalisation of lines corresponding to interior vertices contains precisely one point of valence three.  We refer to the former as {\it type I} lines and the latter as {\it type II} lines.   For a tropical line of type II, suppose that $x+y+z$ is the associated linear form, the three branches of this line are labelled by pairs of variables $(x,y)$, $(y,z)$ and $(x,z)$ according to the pair that realises the initial form along that branch and the point of intersection between these three branches is called the \emph{branch point} of the line. Note that any two such lines {\color{blue} stably} intersect \cite[Section 3.6]{MacStu15} at at most one point. 

\subsection{Proof of Theorem \ref{tropcan_theo}}\label{homeoproof_subsect}

 We construct a homeomorphism $\phi$ between $G^{\rm top}$ and the extended tropicalisation $\mathcal{T}$ of the sch\"on embedding of $X_G$ when $G$ is a three-regular, three-connected planar graph. 
Our strategy is to first construct a bijection between the set of  trivalent points of $G^{\rm top}$, i.e. the set of vertices $V(G)$ of $G$, and the set of trivalent points of $\mathcal{T}$. For this, we use the description of the trivalent points of $\mathcal{T}$ provided  by Lemma \ref{trivalchar_lem}. We define $\phi|V(G)$ as follows. For an interior vertex $u$, we denote the branch point of ${\rm tropproj}(L_u)$ by $b_u$. For an exterior vertex $u$, let $\iota(u)$ denote the unique interior vertex adjacent to it. Recall that for an edge $e=(u,v)$ of $G$, we denote by $\zeta_e$ the (unique) intersection point of ${\rm tropproj}(L_u)$ and ${\rm tropproj}(L_v)$. 
For a vertex $u$ of $G$,

\begin{equation}
\phi(u)=
\begin{cases}
b_u,\text{ if }$u$ \text{ is an interior vertex},\\
\zeta_{(u,\iota(u))}, \text{ if $u$ is an exterior vertex}.
\end{cases}
\end{equation}
By Lemma \ref{trivalchar_lem}, $\phi|V(G)$ is a bijection between $V(G)$ and the set of trivalent points of $\mathcal{T}$. We extend $\phi$ to $G^{\rm top}$ via the following observations. Note that, by definition, $G^{\rm top} \setminus V(G)$ is a disjoint union of open intervals that is in bijection with the edges of $G$.  Consider the set $\mathcal{B}$ of bivalent points of $\mathcal{T}$. By the Bieri-Groves theorem \cite{BieGro84}, \cite[Theorem 3.3.5]{MacStu15}, $\mathcal{B}$ is also a disjoint union of finitely many open intervals. 

Lemma \ref{trivalchar_lem} yields the following description of these open intervals. Recall, from the paragraph ``Tropicalisation of Irreducible Components", that each branch of ${\rm tropproj}(L_u)$, where $u$ is an interior vertex, is labelled by an edge $e$ that is incident on $u$. Consider the set of bivalent points of $\mathcal{T}$ that are contained in a branch of ${\rm tropproj}(L_u)$ where $u$ is an interior vertex.  We denote this set by $\chi_{u,e}$. Consider an exterior vertex $u$, the set ${\rm tropproj}(L_u) \setminus \{\zeta_{(u,\iota(u))}\}$ consists of two connected components. Each connected component is a half-open interval and based on the intersection point that it contains, it corresponds to an edge $e$ of the form $(u,v)$ where $v$ is an exterior vertex. We denote this component by $\chi_{u,e}$. There are three types of open intervals, they are as follows.

\begin{enumerate} 

\item If $e=(u,v)$ where $u$ and $v$ are both interior vertices, then the set $\chi_{u,e} \cup \chi_{v,e}$ is an open interval.  
 
\item If $e=(u,v)$ such that $u$ is an interior vertex and $v$ is an exterior vertex, then $\chi_{u,e}$ is an open interval.

%We denote its set of bivalent points by $\chi_e$.  {\color{blue} 10:00 AM, 28 January, 2022: Give pointer?} 
\item If $e=(u,v)$ where $u$ and $v$ are both exterior vertices, then the set $\chi_{u,e} \cup \chi_{v,e}$ is an open interval.  %This is the first type of open interval and we denote it by $\chi_e$.
 
  \end{enumerate}
  
  We denote each of these three types of open intervals by $\chi_e$ where $e$ is the corresponding edge.   We extend $\phi$ to  $G^{\rm top}$ as follows. Suppose that $\mathcal{I}_e$ is the open line segment in $G^{\rm top}$ corresponding to the edge $e$. We define $\phi|{\mathcal{I}_e}$ to be any homeomorphism between $\mathcal{I}_e$ and $\chi_e$ that when extended to $\bar{\mathcal{I}_e}$ by taking $u$ to $\phi(u)$ and $v$ to $\phi(v)$ induces a homeomorphism between $\bar{\mathcal{I}_e}$ and $\bar{\chi_e}$.  This completes the definition of $\phi$.
  
  Finally, we note that $\phi$ is a homeomorphism between $G^{\rm top}$ and $\mathcal{T}$. We start by noting that, by construction, $\phi$ is a bijection.  Let $N(u)$ be the open neighbourhood $(\cup_{e|u \in e} \mathcal{I}_e) \cup \{u\}$ of the vertex $u \in G^{\rm top}$. Similarly, let $N(\phi(u))$ be the open neighbourhood $(\cup_{e|u \in e} \chi_e) \cup \{\phi(u)\}$ of the trivalent point $\phi(u) \in \mathcal{T}$.  Note that, by construction, the endpoints of $\bar{\chi_e}$ are precisely $\phi(u)$ and $\phi(v)$ for each $e=(u,v)$. Hence, we deduce that $\phi$ induces a homeomorphism between $N(u)$ and $N(\phi(u))$ for every vertex $u$ of $G$.   Consider an open subset $U$ of $\mathcal{T}$ and let $U_u$ be the open subset $U \cap N(\phi(u))$ for a vertex $u$ of $G$. Note that, $\phi^{-1}(U_u)$ is an open subset of $N(u)$ and hence, an open subset of $G^{\rm top}$ for each vertex $u$. Since $\{U_u\}_{ u \in V(G)}$ forms a cover of $U$, we have $\phi^{-1}(U)=\cup_{u} \phi^{-1}(U_u)$. Hence, $\phi^{-1}(U)$ is an open subset of $G^{\rm top}$. 
 This shows that $\phi$ is continuous. The other direction, i.e. that $\phi^{-1}$ is continuous follows exactly analogously. Hence, we conclude that $\phi: G^{\rm top} \rightarrow \mathcal{T}$ is a homeomorphism.   \qed
 
 \begin{remark}

\rm Alternatively, we can also use the fact that a bijective local homeomorphism is a homeomorphism to deduce that $\phi$  is a homeomorphism. \qed
 \end{remark}
 
 We refer to Figure \ref{envelope_fig} for the case when $G$ is the envelope graph. The graph $G$ is shown on the left and the extended tropicalisation of the sch\"on embedding of $X_G$ is shown in thick lines on the right. Note that this tropicalisation is contained in $\mathbb{TP}^3$ which is identified with a three-dimensional simplex. It is contained in two facets of this simplex and these two facets are the visible facets in the figure. 
 The image $\phi(i)$ of $\phi$ on $i$ is denoted by $\phi_i$. The open interval $\chi_{(i,j)}$ is denoted by $\chi_{i,j}$ and is the open interval corresponding to the segment with endpoints $\phi_i$ and $\phi_j$ that appears just below the symbol $\chi_{i,j}$ in the figure.

\begin{figure}
  \centering
  \includegraphics[width=14cm]{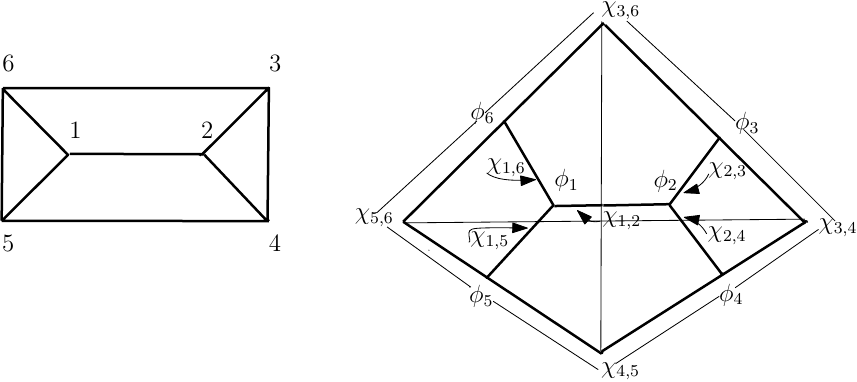}
  \caption{An illustration of the homeomorphism $\phi$ in the case when $G$ is the envelope graph.} \label{envelope_fig}
\end{figure}

%To see this, consider a bivalent point ${\bf p}$ of $\mathcal{T}$. Suppose that ${\bf p}$ is contained in ${\rm tropproj}(L_u)$ for some vertex $u$ (among possibly other tropical lines). If $u$ is an interior vertex, then ${\bf p}$ is contained in a unique branch of ${\rm tropproj}(L_u)$.  By Lemma \ref{trivalchar_lem},  the set of bivalent points on this branch form an open segment. 

\section{Connectivity between Tropicalisations of the Sch\"on Embedding}\label{conn_sect}

In this subsection, we show Theorem \ref{conn_theo} that states that the set of extended tropicalisations of sch\"on embeddings of $X_G$ (as $G$ varies over simple, three-regular, three-connected graphs) is connected via certain ``local'' operations.  
Recall from the introduction that these local operations are  tropical analogues of $\Delta$$Y$, $Y$$\Delta$ and contraction-elongation transformations, and are motivated by Steinitz' theorem.    We start by recalling Steinitz' theorem \cite[Chapter 4]{Zie07} and a part of this standard proof.

\begin{theorem}{\rm ({\bf Steinitz' Theorem})}\label{Stein_thm}
 A  graph is the one-skeleton of a three-dimensional polytope if and only if it is simple, planar and three-vertex-connected.
  \end{theorem}

A key ingredient in the proof is the notion of $\Delta$$Y$ transformation, i.e. replace a $\Delta$-subgraph (a triangular face) by a $Y$-subgraph. More formally, a $\Delta$$Y$ transformation replaces a triangle that bounds a face by a three-star that connects the same set of vertices.  This operation is reversed for a $Y$$\Delta$ transformation. Figure \ref{delta-Y} illustrates the transformations in the case of three-regular graphs, the situation that is relevant for us.

\begin{figure}
  \centering
  \includegraphics[width=8cm]{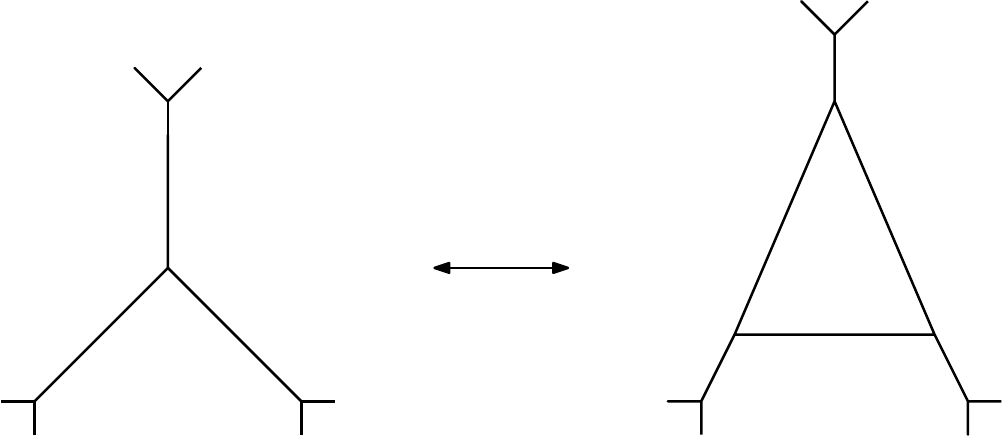}
  \caption{$Y$$\Delta$ and $\Delta$$Y$ transformations.}\label{delta-Y}
\end{figure}

 A simple $Y$$\Delta$ transformation is a $Y$$\Delta$ transformation followed by edge contractions to eliminate valence two vertices and replacing each set of resulting parallel edges by corresponding single edges.    A key step in the proof is to show that every three-vertex-connected planar graph can be obtained from $K_4$ by a sequence of simple $Y$$\Delta$ transformations.

We take cue from this part of the proof. Since we are concerned with \emph{three-regular} planar graphs rather than arbitrary planar graphs, we employ a sequence of ``local operations'' that transform a three-regular, three-connected planar graph to $K_4$ that keep the properties three-regular and three-connected invariant.  We perform the following two operations: 
 
 \begin{enumerate}
\item  $\Delta$$Y$ (and $Y$$\Delta$) transformations.
\item  \emph{Contraction-elongation} transformations, as shown in Figure \ref{elong-contra}. We say that the operation is performed along the edge $e_1$.
\end{enumerate}

\begin{figure}
  \centering
  \includegraphics[width=14cm]{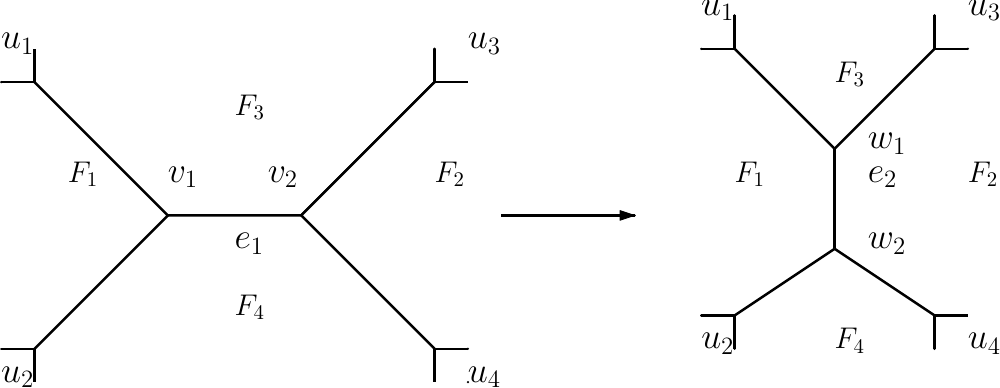}
  \caption{Contraction-elongation transformation.}\label{elong-contra}
\end{figure}

We show that for any three-regular, three-connected planar graph, there is a sequence of $\Delta$$Y$ and contraction-elongation transformations that transforms it to $K_4$ while maintaining the three-regularity and the three-connectivity properties (Lemma \ref{reach_lem}).  We then introduce tropical analogues of $\Delta$$Y$ and contraction-elongation transformations.   We show that any $\Delta$$Y$ transformation on $G$ induces its tropical analogue on the sch\"on embedding of $G$ (Propositions \ref{tropDeltaY_prop} and \ref{tropDeltaY-edge_prop}). Proposition \ref{tropcon-elon_prop} shows an analogous statement for the contraction-elongation transformation.  Theorem \ref{conn_theo} then follows as corollary.  The following characterisation of three-regular, three-connected planar graphs turns out be useful.

%Combining Lemma \ref{reach_lem} with Theorem \ref{tropcan_theo} yields Theorem {\color{blue} give reference}.

 \begin{lemma}\label{threeconnchar_lem}
Let $G$ be a simple, three-regular, two-edge-connected planar graph.  The graph $G$ is three-edge-connected if and only if for any planar embedding of $G$, no two non-adjacent exterior vertices share an interior face. 
 \end{lemma}
 
   \begin{proof}($\Rightarrow$)  Suppose that there is a planar graph $G$ with a planar embedding such that there exist two non-adjacent exterior vertices $u$ and $v$  that share an interior face $F$. Since $G$ is three-regular, for each of these vertices there is a unique edge incident on it  that is shared by both the interior face $F$ and the exterior face.  Suppose that  $e_1$ and $e_2$ are such edges incident on $u$ and $v$, respectively.  We claim that deleting edges $e_1$ and $e_2$ will disconnect the graph.  To see this, consider an arc $\mathcal{A}_1$ with an interior point in $e_1$ and an interior point in $e_2$ as end points and contained in the closure of the exterior face.  Consider another arc $\mathcal{A}_2$ with the same end points contained in the closure of $F$. The interior and exterior regions of the closed curve $\mathcal{A}_1 \cup \mathcal{A}_2$ both intersect $G$ non-trivially, and $\mathcal{A}_1 \cup \mathcal{A}_2$ does not intersect $G \setminus \{e_1,e_2\}$. We conclude that $G \setminus \{e_1,e_2\}$ is disconnected. Hence, $G$ is not three-edge-connected. 

($\Leftarrow$) Conversely, suppose that $G$ is not a three-edge-connected graph. Suppose that deleting edges $e_1$ and $e_2$ disconnects $G$.
  The edges $e_1$ and $e_2$ cannot share a vertex $v$ since this implies that the other edge incident on $v$ is a bridge.  Hence in any planar embedding of $G$, the edges $e_1$ and $e_2$ bind an interior face $F$ of $G$, and are both contained in the exterior face, see Figure \ref{intface_fig}. The vertices $v_1$ and $u_2$ are not adjacent, are both exterior vertices and share an interior face. \qedhere

\end{proof}
\begin{figure}   
  \centering
\includegraphics[width=12cm]{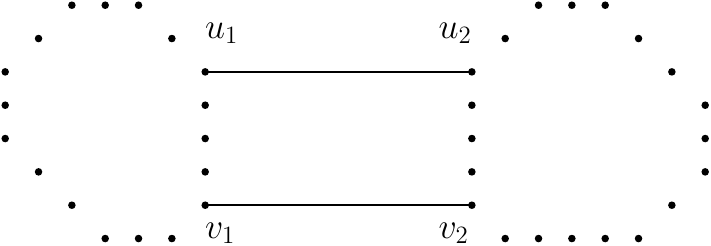}
\caption{Two exterior vertices sharing an interior face.}\label{intface_fig}
 \end{figure}

\begin{figure}
 \centering
  \includegraphics[width=4cm]{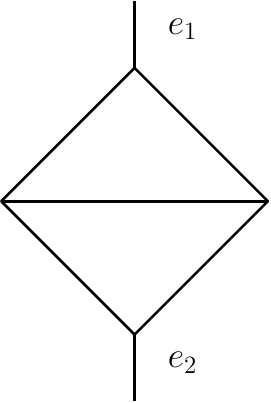}
  \caption{Forbidden subgraph.}\label{forbid-sub}
\end{figure}

\begin{lemma}\label{reach_lem} Every three-regular, three-connected planar graph $G$ can be  transformed to $K_4$ by a sequence of $\Delta$$Y$  and contraction-elongation transformations such that the graph at each step remains a  simple, three-regular, three-connected planar graph. \end{lemma}

\begin{proof} Suppose that $G$ has genus three, then it is a $K_4$ and there is nothing to prove. Otherwise, the genus of $G$ is at least four.  We consider a planar embedding of $G$ and perform the following operations on it.

\begin{enumerate}
\item Suppose that $G$ has a triangular face in this embedding then perform a $\Delta$$Y$transformation  on it.

\item \label{contralong_op} Nevertheless, $G$ has an interior face. Consider an interior face $F$ of the minimum length, $k$ say.  By Item \ref{th:four}, Proposition \ref{grapro_prop}, it has at least $k-1$ interior edges (edges not contained in the exterior face).  Perform a sequence of contraction-elongation transformations on any $k-3$ interior edges. This results in (at least) one  triangular face. 

\item Perform a $\Delta$$Y$  transformation on one of these triangular faces.  
\end{enumerate}

We first show that every graph produced by this procedure is simple, three-regular and three-connected.  Any contraction-elongation transformation along an interior edge does not create a bridge. Suppose it does then this implies that the bridge is the edge $e_2$, see Figure \ref{elong-contra}.  Since if any other edge is a bridge, then the corresponding edge in the original graph must be a bridge.   But if $e_2$ is a bridge, then the original graph can be disconnected by deleting two edges. For instance,  deleting the edges $(u_1,v_1)$ and $(u_3,v_2)$ will disconnect the original graph. Furthermore,  this procedure does not alter the exterior face and the set of faces incident on each exterior vertex remains unaltered. Hence, no two non-adjacent exterior vertices can share an interior face after the operation. Furthermore, the  resulting graph remains simple and three-regular. Hence, by Lemma \ref{threeconnchar_lem}, it remains three-connected. Next, we show that the graph resulting from a $\Delta$$Y$ transformation remains simple, three-regular and three-connected. For a multiple edge to occur from a $\Delta$$Y$ transformation, some two vertices of the $\Delta$ must share a neighbour as shown in Figure \ref{forbid-sub}. But this contradicts the three-connectivity of the graph on which the operation is performed, since deleting the edges $e_1$ and $e_2$ would disconnect the graph.  Hence, the graph remains simple. It remains three-regular by construction and by the proof of Steinitz' theorem \cite[Chapter 4, Lemmas 4.2, 4.2*]{Zie07}, the graph remains three-connected. Hence, after each iteration the resulting graph $G'$ is a simple, three-regular, three-connected graph and its genus $g(G')=g(G)-1$.

We repeat the three operations until the genus of the resulting graph is three, this graph must be a $K_4$ since it is the only  simple, three-regular graph of genus three. 
\end{proof} 

In the following, we define tropical analogues of the notion of $Y$$\Delta$ and contraction-elongation transformations. We define these operations on any tropical line arrangement in tropical projective space $\mathbb{TP}^n$ given some additional data. 

 Our primary example in the current article of such a tropical line arrangement is the extended tropicalisation  $\mathcal{T}$ of the sch\"on embedding of $X_G$ where $G$ is a three-regular, three-connected planar graph.  However, these operations can be carried out in greater generality and can be a topic of future investigation. 
Given a finite set $[0,\dots,n]$, consider tropical projective space $\mathbb{TP}^n$ (identified with the $n$-simplex) each of whose facets are labelled by a distinct element in $[0,\dots,n]$. Suppose we fix the following additional data: for each edge $\xi$ of $\mathbb{TP}^n$, we fix a unique point $\rho_{\xi}$  in the relative interior of $\xi$ that we refer to as the \emph{marked point of $\xi$}.   For each two-dimensional face $D$ of $\mathbb{TP}^n$,  there is a unique branched tropical line \footnote{By a ``branched tropical line'', we mean the one-skeleton of the normal fan of a triangle. Note that we do not impose the balancing and the rational slope conditions. } $\mathcal{TL}_D$ contained in $D$ that passes through $\rho_{\xi}$ for each edge $\xi$ that is contained in $D$. This tropical line is the collection of three rays, corresponding to the three marked points, emanating from the origin in the interior of $D$ (note that the interior of $D$ is identified with $\mathbb{R}^2$ via the extended tropicalisation map).   We refer to $\mathcal{TL}_D$ as the \emph{standard tropical line} associated to $D$. 

%This is the one-skeleton of the normal fan of the triangle such that the points of contact of its edges with its incircle are points in the interior of $D$ corresponding to the three marked points, with its trivalent point at the incentre of the triangle. 

In the following, we identify $\mathbb{TP}^{g-1}$ with a facet of $\mathbb{TP}^g$. Note that given a two-dimensional face $D$ of $\mathbb{TP}^{g-1}$, there is a unique three-dimensional face $\beta_D$ of  $\mathbb{TP}^g$  that contains $D$ and the unique vertex $\nu$ of $\mathbb{TP}^{g}$ that is not contained in $\mathbb{TP}^{g-1}$.

\begin{definition}{\rm ({\bf Tropical $Y$$\Delta$ Transformation at a Two-Dimensional Face})}
A tropical $Y$$\Delta$ transformation of a tropical line arrangement $\mathcal{T} \subset \mathbb{TP}^{g-1} $ at a two-dimensional face $D$ of $\mathbb{TP}^{g-1}$ such that $\mathcal{T}$ contains the standard line of $D$ is the tropical line arrangement
$\hat{\mathcal{T}}:=(\mathcal{T} \setminus \mathcal{TL}_D) \cup \mathcal{TL}_{D^{(1)}} \cup  \mathcal{TL}_{D^{(2)}} \cup  \mathcal{TL}_{D^{(3)}}$ where $D^{(1)},~D^{(2)}$ and $D^{(3)}$ are the two-dimensional faces of $\beta_D$ apart from $D$. 
  \end{definition}
  
We refer to Figure \ref{tropDeltaY_fig} for an illustration of a tropical $Y$$\Delta$ transformation at the face $\{2,3,4\}$. The tropical line $a$ is the standard tropical line of $\{2,3,4\}$ and the tropical lines $b,~c$ and $d$ are the standard tropical lines of $\{1,2,4\},~\{1,2,3\}$ and $\{1,3,4\}$ respectively.

Given an edge $\xi$ of $\mathbb{TP}^{g-1}$, there is a unique two-dimensional face $D_{\xi}$ of $\mathbb{TP}^g$ that contains $\xi$ and the vertex $\nu$ of $\mathbb{TP}^{g}$ not in $\mathbb{TP}^{g-1}$. 

\begin{definition}{\rm ({\bf Tropical $Y$$\Delta$ Transformation at an Edge)}}
A tropical $Y$$\Delta$ transformation of $\mathcal{T}$ at an edge $\xi$ of $\mathbb{TP}^{g-1}$ that is also contained in $\mathcal{T}$ is the tropical line arrangement $\hat{\mathcal{T}}=(\mathcal{T} \setminus \xi) \cup \mathcal{TL}(D_{\xi}) \cup \xi^{(1)} \cup \xi^{(2)}$ where $\xi^{(1)}$ and $\xi^{(2)}$ are the two edges of $D_{\xi}$ apart from $\xi$.
\end{definition}
Figure \ref{tropDeltaY-edge_fig} illustrates a tropical $Y$$\Delta$ transformation at the edge $\xi$, the tropical line $a$ is the standard tropical line of $D_{\xi}$.

For any pair of two-dimensional faces $(D_1,D_2)$ of $\mathcal{TP}^{g-1}$ that shares a common edge,  let $\beta_{D_1,D_2}$ be the unique three-dimensional face of $\mathcal{TP}^{g-1}$ that contains both $D_1$ and $D_2$. 

\begin{definition}{\rm ({\bf Tropical Contraction-Elongation Transformation})}
A tropical contraction-elongation transformation of a tropical line arrangement $\mathcal{T}$ at a pair of two-dimensional faces $(D_1,D_2)$ of $\mathbb{TP}^{g-1}$ that shares a common edge and such that $\mathcal{T}$ contains both $\mathcal{TL}_{D_1}$ and $\mathcal{TL}_{D_2}$ is defined as the tropical line arrangement $\mathcal{T} \setminus (\mathcal{TL}_{D_1} \cup \mathcal{TL}_{D_2}) \cup \mathcal{TL}_{\bar{D_1}}  \cup  \mathcal{TL}_{\bar{D_2}}$ where $\bar{D_1}$ and $\bar{D_2}$ are the two-dimensional faces of $\beta_{D_1,D_2}$ apart from $D_1$ and $D_2$.
\end{definition}
\begin{figure}
\centering  
\includegraphics[width=12cm]{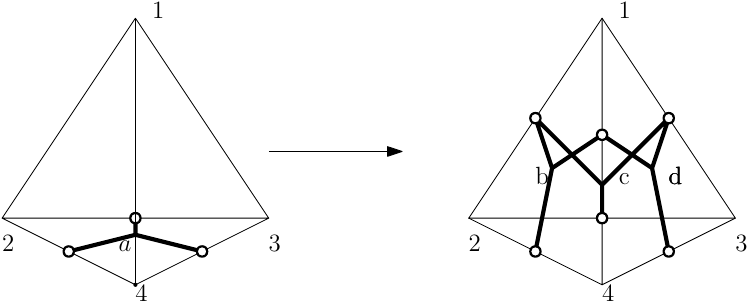}
  \caption{A tropical $Y$$\Delta$ transformation at a two-dimensional face.}\label{tropDeltaY_fig}
\end{figure}

\begin{figure}
\centering  
\includegraphics[width=12cm]{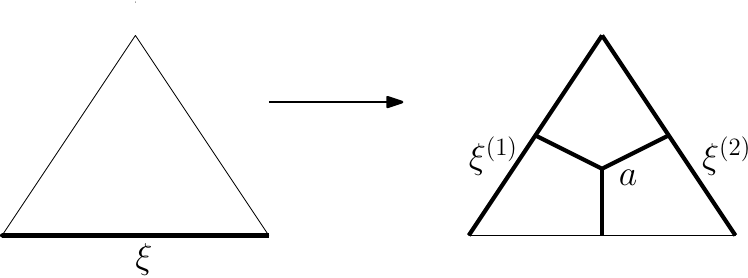}
  \caption{A tropical $Y$$\Delta$ transformation at an edge.}\label{tropDeltaY-edge_fig}
\end{figure}

\begin{figure}
\centering
  \includegraphics[width=12cm]{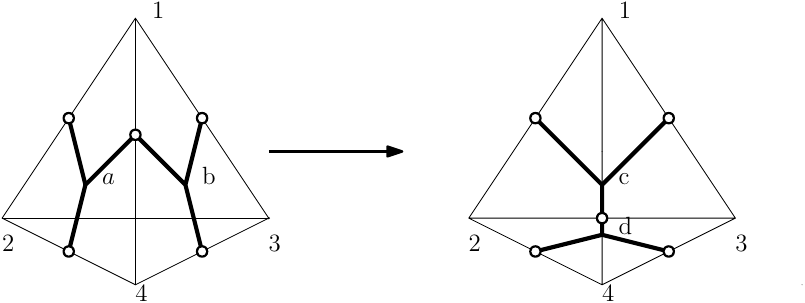}
  \caption{A tropical contraction-elongation transformation.}\label{tropcontraelong_fig}
\end{figure}

We refer to Figure \ref{tropcontraelong_fig} for an example. The tropical lines $a$ and $b$ are the standard tropical lines of the faces $\{1,2,4\}$ and $\{1,3,4\}$ respectively and the tropical lines $c$ and $d$ are the standard tropical lines of $\{1,2,3\}$ and $\{2,3,4\}$ respectively.

In the following, we relate the $Y$$\Delta$ transformation and the contraction-elongation transformation to their tropical analogues.  The set up is as in Section \ref{tropschoen_sect}. For each edge $\xi$ of $\mathbb{TP}^{g-1}$, we set $\rho_\xi:=m_\xi$ (recall its definition from the paragraph ``Tropicalisation of the Irreducible Components'' in Section \ref{tropschoen_sect}). 
 With this choice of marked points, the standard tropical line of each two-dimensional face $D$ of $\mathbb{TP}^{g-1}$ is the extended tropicalisation of the line defined by $\langle x_{F_i}+x_{F_j}+x_{F_k},~x_F|~F \notin \{F_i,F_j,F_k\} \rangle$ where $F_i,F_j$ and $F_k$ are the three facets of $\mathbb{TP}^{g-1}$ not containing $D$.  In the following, we denote the extended tropicalisation of the sch\"on embedding of $X_{G}$ by $\mathcal{T}_G$.  Recall, from Section \ref{tropschoen_sect}, that corresponding to each interior vertex $u$ of $G$, there is a unique two-dimensional face $D_u$ of $\mathbb{TP}^{g-1}$ that contains ${\rm tropproj}(L_u)$.

%Consider a three-regular, three-connected planar graph $G$ with a planar embedding. 
%Consider $\mathbb{TP}^{g-1}$ with each facet labelled by the interior faces of the planar embedding of $G$ 
%along with points $m(\xi)$ for each edge $\xi$ of $\mathbb{TP}^{g-1}$.   

\begin{proposition} \label{tropDeltaY_prop}
Suppose that the graph $G_2$ is the result of a $Y$$\Delta$ transformation of $G_1$ at an interior vertex $u$. The extended tropicalisation $\mathcal{T}_{G_2}$  is the result of a tropical $Y$$\Delta$ transformation of $\mathcal{T}_{G_1}$ at $D_u$ with respect to the marked points $\{m_{\xi}\}_{\xi}$.
\end{proposition}
\begin{proof}
Note that $\mathcal{T}_{G_2}=\mathcal{T}_{G_1} \setminus \{ {\rm tropproj}(L_{u})\} \cup {\rm tropproj}(L_{u^{(1)}}) \cup  {\rm tropproj}(L_{u^{(2)}}) \cup {\rm tropproj}(L_{u^{(3)}})$ where $u^{(1)},u^{(2)}$ and $u^{(3)}$ are the three vertices of $\Delta$, the new face that is created by the transformation. Consider the three-dimensional subsimplex $\beta$ of $\mathbb{TP}^g$ whose label is the complement of the set $\{F_i,F_j,F_k,\Delta\}$ where $F_i,F_j$ and $F_k$ are the three interior faces of $G_1$ that contain $u$. The two-dimensional face $D_u$ is a face of $\beta$. The proof follows from the observation that the other three two-dimensional faces of $\beta$ are precisely $D_{u^{(i)}}$ for each $i$ from one to three. 
\end{proof}
Recall that for an exterior vertex $u$ of $G$, the tropical line ${\rm tropproj}(L_u)$ coincides with an edge of $\mathbb{TP}^{g-1}$. 
\begin{proposition} \label{tropDeltaY-edge_prop}
Suppose that the graph $G_2$ is the result of a $Y$$\Delta$ transformation of $G_1$ at an exterior vertex $u$. The extended tropicalisation $\mathcal{T}_{G_2}$  is the result of a tropical $Y$$\Delta$ transformation of $\mathcal{T}_{G_1}$ at the edge ${\rm tropproj}(L_u)$ with respect to the marked points  $\{m_{\xi}\}_{\xi}$. 
\end{proposition}
The proof of Proposition  \ref{tropDeltaY-edge_prop} is analogous to the proof of Proposition \ref{tropDeltaY_prop}.  

\begin{proposition}\label{tropcon-elon_prop}
Suppose that the graph $G_2$ is the result of a contraction-elongation transformation of $G_1$ along the edge $e=(u,v)$ (both $u$ and $v$ are interior vertices). The extended tropicalisation $\mathcal{T}_{G_2}$  is the result of a tropical contraction-elongation transformation of $\mathcal{T}_{G_1}$ at the pair $(D_u,D_v)$ with respect to the marked points $\{m_{\xi}\}_{\xi}$.
\end{proposition}
\begin{proof}
The proof is analogous to the proof of Proposition \ref{tropDeltaY_prop}. Suppose that $F_i,~F_k$ are the two  faces incident on $e$ and that $F_j$ ($F_l$, respectively) is the other face incident on $u$ ($v$, respectively). Consider the three-dimensional subsimplex $\beta$ of $\mathbb{TP}^{g-1}$ that is labelled by the complement of the set $\{F_i,F_j,F_k,F_l\}$. Suppose that $D_i,D_j,D_k$ and $D_l$ are the four two-dimensional faces of $\beta$ defined by the property that the label of $D_r$ additionally contains $F_r$ for $r \in \{i,j,k,l\}$.  The proof follows from the observation that ${\mathcal T}_{G_2}=\mathcal{T}_{G_1}\setminus ({\mathcal{TL}}(D_j) \cup {\mathcal{TL}}(D_l) ) \cup (\mathcal{TL}(D_i) \cup \mathcal{TL}(D_k))$.
\end{proof}

\begin{definition}
Tropical line arrangements $\mathcal{T}_1$ and $\mathcal{T}_2$ in tropical projective space are said to be  \emph{related by a tropical $\Delta$$Y$ transformation} if one of them, $\mathcal{T}_2$ say, is a tropical $Y$$\Delta$ transform of the other. 
We say that $\mathcal{T}_1$ is a tropical $\Delta$$Y$ transform of $\mathcal{T}_2$.  
\end{definition}

As a corollary to Propositions \ref{tropDeltaY_prop} and \ref{tropDeltaY-edge_prop}, we obtain:
\begin{corollary}
Suppose that the graph $G_1$ is a $\Delta$$Y$ transform of $G_2$. The extended tropicalisation $\mathcal{T}_{G_1}$ is a tropical $\Delta$$Y$ transform of $\mathcal{T}_{G_2}$ with respect to the marked points $\{m_{\xi}\}_{\xi}$. 
\end{corollary}

%, i.e. the face that is labelled by the complement of the set of interior faces of $G_1$ that contain $v$.

%Consider a finite arrangement of tropical lines {\color{blue} give reference} in $\mathbb{TP}^n$ such that each of its component tropical lines either has a unique branch point of valence three or is unbranched.  

% Recall, from the previous subsection, that for an interior vertex $v$ of $G$, the tropical line ${\rm tropproj}(L_v)$ is contained in a unique two-dimensional face $D_v$ of $\mathbb{TP}^{g-1}$. Furthermore, any  two-dimensional face $D$ of $\mathbb{TP}^{g-1}$ is labelled by a set $\mathfrak{L}$ of $(g-3)$ interior faces and contains the extended tropicalisation of the line $\langle x_{F_i}+x_{F_j}+x_{F_k},x_F \rangle$ where $\{F_i,F_j,F_k\}$ are the interior faces that are not contained in $\mathfrak{L}$ and $F$ varies over all faces in $\mathfrak{L}$. We denote this tropical line by $\mathcal{TL}_D$. In the following definition, we identify $\mathbb{TP}^{g-1} $ with a facet of $\mathbb{TP}^{g}$.

%Given an extended tropicalisation  $\mathcal{T}$ of the sch\"on embedding of $X_G$ where $G$ is a three-regular, three-connected planar graph.

As a corollary to Lemma \ref{reach_lem} and the correspondence between  $\Delta$$Y$, $Y$$\Delta$, contraction-elongation transformations and their respective tropical analogues, we obtain Theorem \ref{conn_theo}.   %This leads to the following connectivity property of the set $\{\mathcal{T}_G\}_G$. 

%\begin{theorem} Let $G_1$ and $G_2$ be three-regular, three-connected planar graphs. There exists a finite sequence consisting of tropical $\Delta$$Y$, tropical $Y$$\Delta$ and tropical contraction-elongation transformations that transforms   $\mathcal{T}_{G_1}$ to $\mathcal{T}_{G_2}$. \end{theorem}

\begin{figure}
\centering
  \includegraphics[width=12cm]{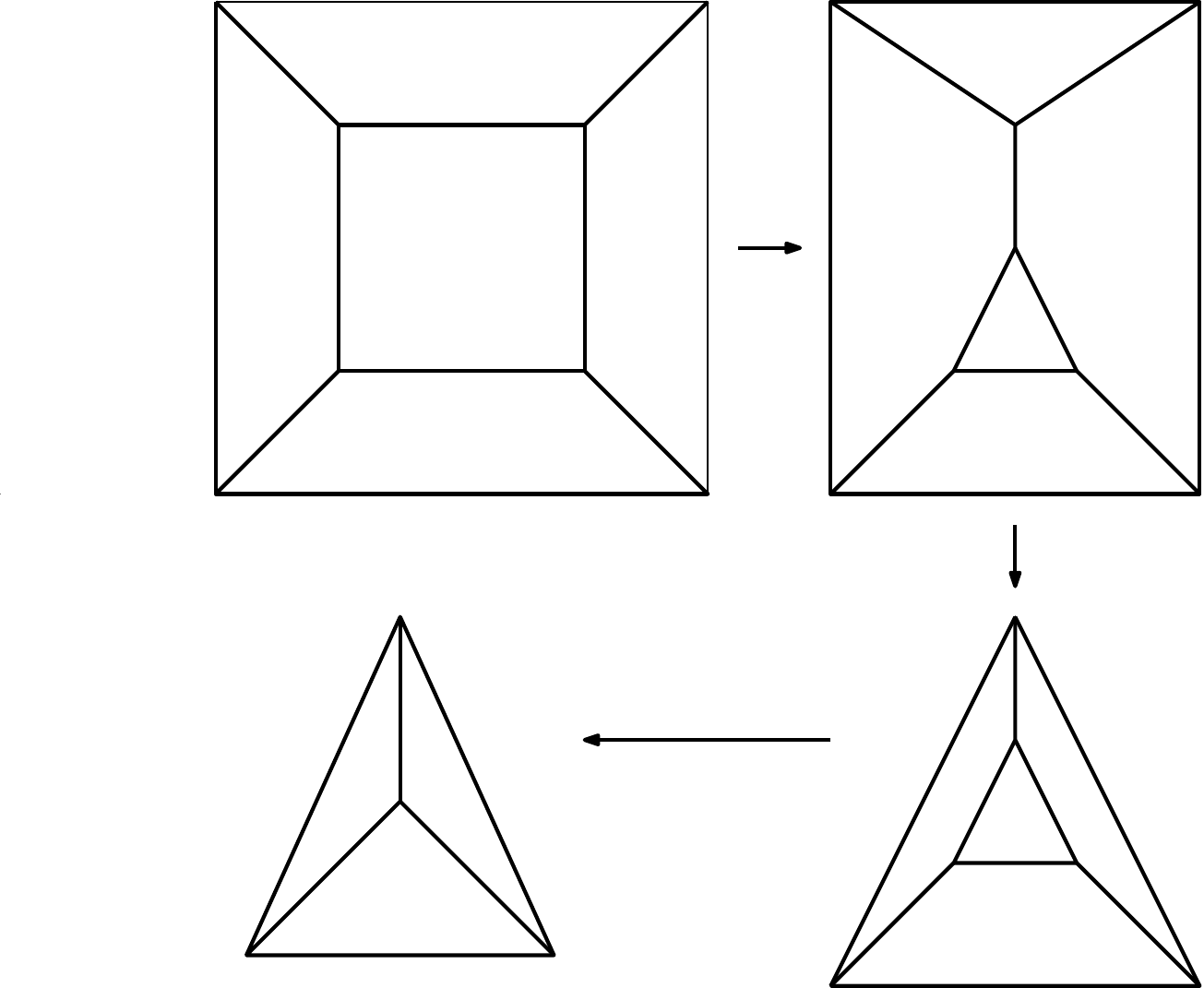}
  \caption{A sequence of $Y$$\Delta$  and contraction-elongation transformations.}\label{excongra_fig}
\end{figure}

\begin{figure}
\centering
  \includegraphics[width=12cm]{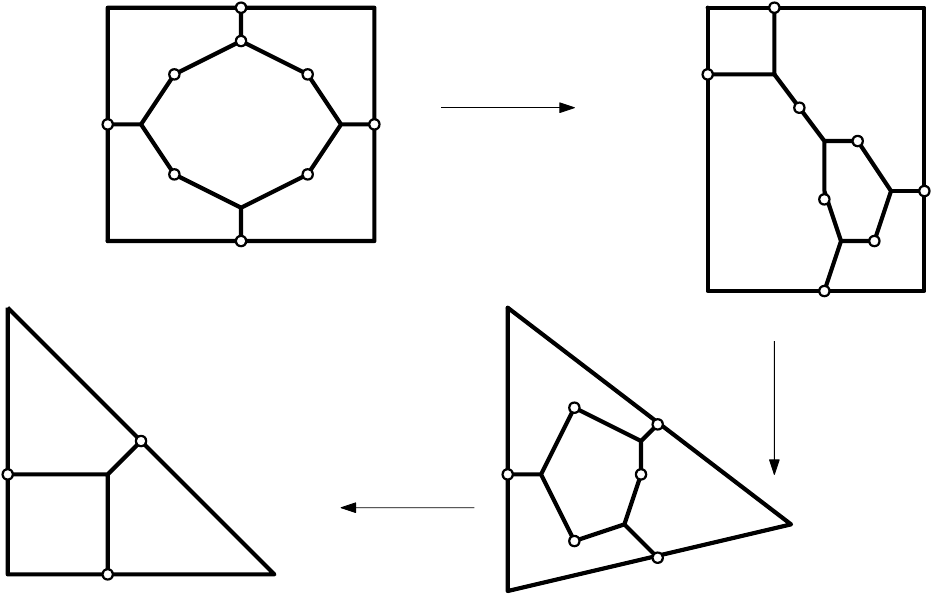}
  \caption{The corresponding tropical operations.}\label{exconn_fig}
\end{figure}

\begin{example}\label{cube_ex} \rm Consider the one-skeleton of the three-dimensional cube. A sequence of transformations following Lemma \ref{reach_lem} to transform it into $K_4$  is shown in Figure \ref{excongra_fig}. This sequence is the one-skeleton of the following polytopes: 

\begin{center}cube $\rightarrow$  triangular prism sliced at a vertex $\rightarrow$  triangular prism $\rightarrow$ tetrahedron. \end{center}

The tropical counterpart of the sequence is shown in Figure \ref{exconn_fig}. \qed

\end{example}

\footnotesize
\noindent {\bf Author's address:}

\smallskip

\noindent Department of Mathematics,\\
Indian Institute of Technology Bombay,\\
Powai, Mumbai,
India 400076.\\
Email: madhu@math.iitb.ac.in, madhusudan73@gmail.com

\end{document}